\documentclass[12pt,a4paper,oneside,final,reqno]{amsart}

\usepackage{amsthm}
\usepackage{amsmath}
\usepackage{amssymb}
\usepackage{mathtools}
\usepackage{amscd}
\usepackage[utf8]{inputenc}
\usepackage{typearea}

\usepackage{yfonts}
\usepackage{textcomp}
\usepackage{mathrsfs}
\usepackage{hyperref}
\usepackage[draft]{fixme} % erase draft to disable fixme notes
\usepackage{pdfsync}
\usepackage[active]{srcltx}
\usepackage{verbatim}
\usepackage{todonotes}
\usepackage[capitalise]{cleveref}
\usepackage{tikz-cd}
\usepackage{tensor}

\usepackage[backend=bibtex, style=alphabetic, maxnames=50]{biblatex}

\newcommand{\N}{\mathbb{N}}

\newcommand{\Go}{G^{(0)}}

\def\acts{\curvearrowright}

\theoremstyle{plain}
\newtheorem{theorem}{Theorem}[section]
\newtheorem{corollary}[theorem]{Corollary}
\newtheorem{lemma}[theorem]{Lemma}
\newtheorem{proposition}[theorem]{Proposition}

\newtheorem{introtheorem}{Theorem}

\newtheorem{introcorollary}[introtheorem]{Corollary}

\theoremstyle{definition}
\newtheorem{definition}[theorem]{Definition}
\newtheorem{example}[theorem]{Example}
\newtheorem{remark}[theorem]{Remark}

\numberwithin{equation}{section}
 
\title{Growth Conditions for Étale Groupoids: Amenability and Comparison}

\date{\today}

\author[1]{Are Austad}
\address{Are Austad, Department of Mathematics, University of Oslo, P.O. Box 1053 Blindern, 0316 Oslo, Norway}
\email{areaus@math.uio.no}

\author[2]{Christian Bönicke}
\address{Christian Bönicke, School of Mathematics, Statistics and Physics, Newcastle University, Newcastle upon Tyne NE1 7RU,
UK}
\email{christian.bonicke@ncl.ac.uk}

\keywords{\'Etale groupoids, amenability, dynamical comparison, length functions, polynomial growth}
\subjclass[2020]{Primary: 22A22; Secondary: 37B05, 43A07}
\begin{document}
\begin{abstract}
   We investigate \'etale Hausdorff groupoids whose source fibres have controlled growth with respect to a length function. Our first result shows that uniform subexponential growth implies topological amenability.  
    If the unit space is compact and metrizable and such a length satisfies a suitable uniform doubling condition, then the groupoid has weak $m$-comparison for some $m$. 
    For minimal ample groupoids, this implies comparison and almost unperforation of the type semigroup. Under the additional assumption that the unit space has no isolated points, we deduce Matui’s AH-conjecture.
\end{abstract}

\maketitle

\section{Introduction}
 The growth rate of a finitely generated group, that is the asymptotic behaviour of the function counting the number of elements of length at most $n$ (relative to a finite symmetric generating set), has long been a central object of interest in geometric group theory.  
 A landmark theorem by Gromov \cite{GromovPolyGrowth81} tells us that a finitely generated group has polynomial growth if and only if it admits a nilpotent subgroup of finite index. As a consequence, finitely generated groups with polynomial growth automatically satisfy a formally much stronger property: bounded doubling. 
 Thus, groups of polynomial growth form a particularly well-behaved class of amenable groups.

 The notion of a length function generalises from discrete groups to \'etale groupoids in a straightforward manner (cf. Definition \ref{def:length-function}), allowing us to consider the growth rate of an \'etale groupoid $G$ with respect to a length function $\ell \colon G \to [0,\infty)$. 
 We call $\ell$ \emph{uniformly fibrewise coarse} when it is bounded on compact sets and its source fibre balls have uniformly bounded cardinalities at each fixed radius. Growth bounds are uniform over the units, see Definition~\ref{def:gpd-poly-growth}.
 With the growing number of applications of topological groupoids in dynamics, noncommutative geometry, operator algebras, and group theory (via the construction of the associated topological full groups)
 it seems natural to investigate whether imposing a constraint on the growth rate has any structural consequences for the groupoid. While there is no analogue of Gromov's theorem available for groupoids, one might still hope to use the machinery of growth rates to produce strong results. 

 This idea was first executed successfully by Nekrashevych, who used it to construct simple finitely generated algebras of arbitrary Gelfand-Kirillov dimension $\geq 2$ and simple finitely generated algebras of quadratic growth over arbitrary fields \cite{Nekrashevych2016}. 
 More recently, \cite{Hou2017} and \cite{AOP-JNCG} showed that certain operator algebras associated with \'etale groupoids with polynomial growth admit spectrally invariant subalgebras of rapidly decaying functions. 

 Passing to the world of topological groupoids, the notion of amenability splits into two separate concepts that are independent of each other: topological amenability and fibrewise amenability.  
 Topological amenability is very natural from a $C^*$-algebraic perspective, as it turns out to be equivalent to nuclearity of the associated reduced groupoid $C^*$-algebra \cite{DelarocheRenaultAmenableGpds2000}.
 Fibrewise amenability on the other hand is based on 
 F{\o}lner conditions on finite subsets of the fibres  
 (see \cite{MaWu2020,Ma2021} for a more detailed discussion of this concept). While both concepts agree for discrete groups, in the realm of \'etale groupoids there is no implication between the two in general. A groupoid can be topologically amenable but not fibrewise amenable, and vice versa. 

 The study of topological amenability for groupoids has received a lot of attention, and as such there are many examples of topologically amenable \'etale groupoids in the literature, e.g.\ AF groupoids, graph groupoids, Deaconu-Renault groupoids \cite{SimsWilliamsPrimitiveIdeals2016}, transformation groupoids arising from amenable actions, and coarse groupoids arising from spaces with property (A) \cite{SkandalisTuYu2002}, though this is not an exhaustive list.

 A standard proof that groups with subexponential growth are necessarily amenable relies on the fact that the balls along a suitable sequence of increasing radii
 are F{\o}lner sets for the group. Thus, it is not surprising that an adaptation of this proof to the setting of groupoids shows that an \'etale groupoid with subexponential growth is fibrewise amenable in the sense of \cite{MaWu2020}, see \cite[Proposition 3.2]{AOP2026}. Our first main result tells us that subexponential growth in fact also implies topological amenability.

 \begin{introtheorem}[cf. Theorem \ref{thm:subexp-growth-implies-amenable}]
     Let $G$ be a 
     locally compact Hausdorff \'etale groupoid, and suppose $G$ admits a uniformly fibrewise coarse length function with subexponential growth. Then $G$ must be topologically amenable. 
 \end{introtheorem}

 A similar result has long been known in the measurable setting (see \cite[Proposition~3.2.32]{DelarocheRenaultAmenableGpds2000}). 
 Moreover, for compactly generated ample groupoids, the analogous result was proved by Nekrashevych \cite[Proposition 3.14]{Nekrashevych25}. 
 Here we produce a concrete approximation satisfying Reiter's condition in the presence of a continuous length function on $G$ which has subexponential growth. A series of reductions then implies the general statement above.

When an \'etale groupoid $G$ has uniformly bounded doubling with respect to a uniformly fibrewise coarse length function $\ell$ (meaning that  doubling the radius increases each source fibre ball cardinality by at most a fixed factor, see Definition~\ref{def:gpd-poly-growth}), weighted averages over the balls $B_\ell (R)$ define upper and lower Ces\`aro densities. For compact metrizable unit space, weak$^*$ limit points of these averages yield $G$-invariant probability measures, and the density estimates used below. In light of this it becomes natural to wonder whether such groupoids satisfy comparison. 

Recall that heuristically, a subset $A\subseteq \Go$ is subequivalent to a subset $B\subseteq \Go$, written $A\precsim B$, if one can cut $A$ into finitely many pieces and translate each piece using the dynamics of $G$ so that the images are disjoint subsets of $B$. 
We say that an \'etale groupoid $G$ has comparison, if this subequivalence relation is determined by the invariant positive regular Borel probability measures on the unit space $\Go$. On a technical level, it is often easier to verify a variant of comparison called weak $m$-comparison. We refer the reader to Definition \ref{def:dynamical-comparison} for precise definitions.

The additional \emph{comparison admissibility} hypothesis means that every sufficiently large ball $B_\ell(R)$ has a finite cover by compact open bisections contained in $B_\ell(\Lambda R)$, for a fixed $\Lambda\geq 1$ (Definition \ref{def:lf-admissibility}).

\begin{introtheorem}[cf. Theorem \ref{thm:main-theorem-wocg}]\label{introtheorem:pg-implies-dyn-comp}
   Suppose $G$ is a locally compact Hausdorff \'etale groupoid with compact and metrizable unit space.
     If $G$ admits a comparison admissible uniformly fibrewise coarse length function with uniformly bounded doubling,
     then
     there exists an $m \in \N \cup \{0\}$ such that $G$ has weak $m$-comparison. 
\end{introtheorem}

The comparison property has many desirable consequences.
In the context of transformation groupoids for example, Downarowicz and Zhang \cite{DownarowiczZhang23} relate the comparison property to the problem of a lossless digitalization of a dynamical system by extending it to a subshift on finitely many symbols. 
The comparison property is also popular among the $C^*$-algebra community, where it serves as a dynamical analogue of the strict comparison property appearing in the structure theory of nuclear $C^*$-algebras. In this context, Kerr's work \cite{KerrDimensionComparisonAlmostFiniteness} promoted the idea to develop a dynamical analogue of the Toms-Winter conjecture. Further work in this direction has subsequently been carried out by a variety of authors, see for example \cite{KerrSzaboAlmostfinite2020,LiaoTikuisisAFComparisonTracialZstability,KopsacheilisWinter}. 
Verifying the comparison property is often non-trivial. Downarowicz and Zhang \cite{DownarowiczZhang23} showed that all actions of groups with locally subexponential growth on zero-dimensional compact spaces have comparison. This was recently extended to actions of amenable groups on zero-dimensional compact spaces in \cite{GlasnerLiu26}.
Naryshkin \cite{NaryshkinComparison2022} showed that actions of finitely generated groups of polynomial growth on compact metrizable spaces have weak $m$-comparison. 
Our result extends \cite{NaryshkinComparison2022} to \'etale groupoids admitting a comparison admissible uniformly fibrewise coarse length function with uniformly bounded doubling. The comparison bound depends only on the doubling and admissibility constants. A very recent complementary result in the purely infinite regime was provided by \cite{GHK26}.

Matui's AH-conjecture \cite[Conjecture 2.9]{MatuiAHHK16} proposes a relationship between the first two homology groups of a minimal ample  groupoid and the abelianisation of its topological full group. 
Li showed that one should interpret the AH-conjecture as the degree one component of a much deeper relation between homological invariants of topological full groups and the homology of ample groupoids, and verified this relationship (and in particular the AH-conjecture) for minimal ample groupoids with comparison and unit spaces without isolated points \cite{LiInfLoopSpaces25}.   Here the topological full group is understood as the group of compact open full bisections. For effective groupoids this agrees with its usual realization by homeomorphisms of the unit space.
By adapting results from \cite{AraBonickeBosaLi23} and combining them with
Theorem \ref{introtheorem:pg-implies-dyn-comp}, we deduce that minimal ample groupoids satisfying the length hypotheses of Theorem \ref{introtheorem:pg-implies-dyn-comp} fall within the purview of the main results in \cite{LiInfLoopSpaces25}. In particular, we obtain the following result. 

\begin{introcorollary}[cf. Corollary \ref{cor:AH-conjecture}]
Let $G$ be a locally compact Hausdorff
    minimal ample groupoid whose unit space is compact and  metrizable and has no isolated points. 
    If $G$ admits a comparison admissible uniformly fibrewise coarse length function with uniformly bounded doubling, then $G$ satisfies Matui's AH-conjecture.  
\end{introcorollary} 

The applications in Section \ref{sec:comparison-and-examples} include semi-saturated partial actions of $\mathbb Z^d$ with clopen domains and groupoids associated with Delone sets in locally compact groups whose length balls have bounded Haar volume doubling. In the latter case, finite local complexity and minimality of the hull action give comparison and almost unperforation of the type semigroup, while aperiodicity and noncompactness of the ambient group give the absence of isolated units required for the AH-conjecture.

\subsubsection*{Acknowledgments} The second author would like to thank Hannes Thiel for suggesting that we consider comparison in the context of groupoids with polynomial growth. 
We are also indebted to Jean Renault for pointing out an error in an earlier version of this work.
The first author wishes to thank Newcastle University for its hospitality during a research stay in February 2026, during which parts of this research were conducted. 
The first author was funded by the Research Council of Norway projects 324944 and 359602.

\section{\texorpdfstring{\'Etale}{Étale} groupoids, length functions and growth}
We begin by  
discussing length and growth of \'etale groupoids. We refer the reader to \cite{SimsSzaboWilliams2020} for the basic notions on groupoids.

Let $G$ be a groupoid with unit space $\Go$, and denote by $G^{(2)}$ its composable pairs. The range and source maps $r,s \colon G \to \Go$ will be denoted $r$ and $s$, respectively. 
We denote the fibres of the range and source maps by $G^x = r^{-1}(x)$ and $G_x = s^{-1}(x)$ for $x \in \Go$. 
Furthermore, for subsets $A, B \subseteq G$, we denote the set of all possible products of elements in $A$ and elements in $B$ by $AB = \{\gamma\mu \in G \mid \gamma \in A, \mu \in B, (\gamma,\mu) \in G^{(2)}\}$. In the special case where $B = \{x\}$ for some $x \in \Go$, we merely write $Ax$ for simplicity.

We say $G$ is a \emph{topological} groupoid if it is equipped with a topology such that the product, inverse, range and source maps are continuous. Again, we equip $G^{(2)}$ with the relative topology from $G \times G$, and $\Go$ with the topology  induced by $G$. Any topological groupoid we consider in this article will be locally compact and Hausdorff.

A locally compact Hausdorff groupoid $G$ is said to be \'etale if the range map $r$ (equivalently, the source map $s$) is a local homeomorphism onto $\Go$. In particular, $G^x$ and $G_x$ are discrete with respect to the subspace topology. When $G$ is \'etale, there is a canonical Haar system on $G$ given by the counting measure on $G^x$ for all $x\in G^{(0)}$. 

Invariance for the Haar system is typically only stated for integrating compactly supported functions. The following lemma records the corresponding reindexing identity for arbitrary non-negative functions. 

\begin{lemma} \label{Lemma: invariance}
Let $G$ be an \'etale groupoid and let
$(\lambda^x)_{x\in G^{(0)}}$ denote the Haar system of counting measures.
Let $f:G\to [0,\infty]$ be any non-negative function. Then, for every
$\alpha\in G$,
\[
\sum_{\eta\in G^{r(\alpha)}} f(\eta)
=
\sum_{\gamma\in G^{s(\alpha)}} f(\alpha\gamma).
\]
In particular, if $f\in C_0(G)^+$ and
$\sup_{x\in G^{(0)}}\lambda^x(f)<\infty$, then both sides are finite and
the same equality holds.
\end{lemma}

\begin{proof}
Left multiplication by $\alpha$ gives a bijection
\[
G^{s(\alpha)} \longrightarrow G^{r(\alpha)},
\qquad
\gamma\mapsto \alpha\gamma,
\]
with inverse given by left multiplication by $\alpha^{-1}$. Reindexing the
counting sum by this bijection gives
\[
\sum_{\eta\in G^{r(\alpha)}} f(\eta)
=
\sum_{\gamma\in G^{s(\alpha)}} f(\alpha\gamma).
\]
Since $f\geq 0$, the equality is valid in $[0,\infty]$. If moreover
$\sup_{x\in G^{(0)}}\lambda^x(f)<\infty$, then the sums are finite.
\end{proof}

We proceed to discuss length functions and growth of \'etale groupoids. 
\begin{definition}\label{def:length-function}
    Let $G$ be  
    a locally compact Hausdorff 
    \'etale groupoid. We say that a function $\ell \colon G \to [0, \infty)$ is a \emph{length function on }$G$ if
    \begin{enumerate}
        \item $\ell(\gamma) = 0$ for all $\gamma \in \Go$,
        \item $\ell (\gamma^{-1}) = \ell(\gamma)$ for all $\gamma \in G$, and
        \item $\ell(\gamma \mu) \leq \ell(\gamma) + \ell(\mu)$ for all $(\gamma, \mu) \in G^{(2)}$. 
    \end{enumerate}
For any $t \geq 0$, we denote by $B_\ell(t)$ the set
\begin{align*}
    B_\ell(t) := \{ \gamma \in G \mid \ell(\gamma) \leq t \}.
\end{align*}
  Then $\ell$ is said to be \emph{controlled} if $\sup_{\gamma\in K}\ell(\gamma)<\infty$ for all precompact sets $K\subseteq G$, and we say that $\ell$ is 
  \emph{uniformly fibrewise proper} if 
  \begin{align*}
      \sup_{x \in \Go} \vert B_\ell(t) x \vert < \infty 
  \end{align*}
  for all $t \geq 0$. 
  We say that $\ell$ is a  
  \emph{uniformly fibrewise coarse} 
  length function if it is both controlled and uniformly fibrewise proper.
\end{definition}
    
\begin{remark}
    Throughout this article, a proper length function means a length function which is a proper map to $[0,\infty)$. For a continuous length this is equivalent to compactness of every $B_\ell (R)$. A continuous proper length function is uniformly fibrewise coarse: Continuity gives controlledness, and a compact ball is covered by finitely many open bisections, each meeting a source fibre in at most one arrow. 
    This convention differs from the coarse-length terminology in \cite[Definition 3.3]{Ma2021}, which requires bounded subsets of the non-unit part to be precompact. 
\end{remark}

Given a uniformly fibrewise coarse length function $\ell:G\to [0,\infty)$ on 
a locally compact Hausdorff 
\'etale groupoid, there is an associated \emph{growth function} $\gamma_\ell \colon [0,\infty) \to [0,\infty)$ given by
 \begin{align}\label{eq:growth-function}
     \gamma_\ell(t) &= \sup_{x \in \Go} \vert B_\ell(t)x \vert. 
 \end{align}

 \begin{definition}\label{def:gpd-poly-growth}
    Let $G$ be a locally compact Hausdorff \'etale groupoid and let $\ell:G\to [0,\infty)$ be a uniformly fibrewise coarse length function on $G$. We say that 
    $(G,\ell)$ has 
    \begin{itemize}
        \item \emph{uniformly bounded doubling} if there exists a constant $C\geq 1$ such that
        $$\vert B_\ell(2t)x\vert \leq C\vert B_\ell(t)x\vert\text{ for all }t\geq 0,\text{ and for all } x\in G^{(0)},$$
        \item \emph{polynomial growth} if there exists a $C\geq 0$ and $d\in \mathbb N$ such that
        $$\gamma_\ell(t)\leq C (t+1)^d \text{ for all }t\geq 0,$$
        \item \emph{subexponential growth} if for every $a>0$ there exists a $C_a\geq 0$ such that
        $$\gamma_\ell(t)\leq C_a e^{at} \text{ for all } t\geq 0.$$
    \end{itemize}
 \end{definition}
We note that for a fixed length function $\ell$ on a locally compact Hausdorff \'etale groupoid $G$ uniformly bounded doubling implies polynomial growth and polynomial growth implies subexponential growth. The reverse implications do not hold in general. When we say that $G$ has one of these growth properties, we mean that it admits a uniformly fibrewise coarse length function $\ell$ such that $(G,\ell)$ has the respective growth property. 
Note that if $H \subseteq G$ is an open subgroupoid and $G$ has polynomial (resp. subexponential) growth, then the restricted length has at most the same growth on $H$. In particular, it has polynomial (resp. subexponential) growth. Controlledness passes to $H$ with its subspace topology.

\section{Subexponential growth implies amenability}\label{sec:pg-implies-amenability}
In this section we prove that any \'etale groupoid admitting a uniformly fibrewise coarse length function $\ell$ for which $\gamma_\ell$ has subexponential growth is topologically amenable. We first define topological amenability. There are many equivalent definitions of this notion, see \cite{DelarocheRenaultAmenableGpds2000} for an extensive treatment, but we will first and foremost have use for the following condition, known as \emph{Reiter's condition}, see \cite[Proposition~2.2.13]{DelarocheRenaultAmenableGpds2000}.

\begin{definition}\label{def:amenability-through-Reiter}
    Let $G$ be a locally compact Hausdorff \'etale groupoid. Let $(\lambda^x)_x$ denote the continuous Haar system of counting measures. We say $G$ is \emph{amenable} if there exists a net $(a_i)_i$ in $C_c(G)^+$ such that
\begin{enumerate}
    \item $\sup_{x\in G^0}\lambda^x(a_i)\leq 1$ for all $i$,
    \item $\lim_i \lambda^x(a_i)=1$ uniformly on compact subsets of $G^{(0)}$, and
    \item  $\lim_i \sup_{\alpha\in K}\sum_{\gamma\in G^{s(\alpha)}} \vert a_i(\alpha\gamma)-a_i(\gamma)\vert =0$ for every compact subset $K\subseteq G$.
\end{enumerate}
\end{definition}

Before proceeding with the proof of Theorem \ref{thm:subexp-growth-implies-amenable} below, we will make some simplifications.

\begin{lemma}\label{lemma:sigma-cpct-amen-suffices}
    Every compact subset of a locally compact Hausdorff \'etale groupoid is contained in an open $\sigma$-compact subgroupoid. Consequently, a locally compact Hausdorff \'etale groupoid is  amenable if and only if every open $\sigma$-compact subgroupoid of $G$ is amenable. 
\end{lemma}

\begin{proof}
    Let $K\subseteq G$ be compact. Choose a compact symmetric set $C_1$ containing
$K\cup K^{-1}\cup r(K)\cup s(K)$. Inductively, suppose that compact symmetric sets $C_1,\ldots, C_n\subseteq G$ have been chosen such that $C_i\subseteq \operatorname{int}(C_{i+1})$ and $C_iC_i\subseteq C_{i+1}$ for all $1\leq i\leq n-1$.
The compact set $C_n\cup C_nC_n$ is contained in an open precompact symmetric set
$O_n$. We may arrange this by intersecting any open precompact neighbourhood of $C_n\cup C_nC_n$ with its inverse, noting that since $C_n\cup C_nC_n$ is symmetric, it will be contained in the resulting intersection.
Now put
$C_{n+1}:=\overline{O_n}$. Then $C_n\subseteq\operatorname{int}(C_{n+1})$,
$C_{n+1}$ is compact and symmetric, and
\[
H:=\bigcup_{n\geq1}C_n
\]
is open and $\sigma$-compact. It is closed under inversion and multiplication: if two
composable elements lie in $C_n$, their product lies in $C_nC_n\subseteq C_{n+1}$.
Moreover $\gamma^{-1}\gamma$ and $\gamma\gamma^{-1}$ belong to $C_nC_n$ for every $\gamma\in C_n$, so $H$ contains the required units and is a subgroupoid.

For the second part, suppose every open $\sigma$-compact subgroupoid of $G$ is amenable. Fix $K\subseteq G$ and $J \subseteq G^{(0)}$ compact, and an $\varepsilon > 0$. We now apply the first part of the proof to the compact set $K \cup J \cup s(K) \cup r(K)$, thus obtaining an open $\sigma$-compact subgroupoid $H \subseteq G$ with $K \cup J \cup s(K) \cup r(K) \subseteq H$. As $H$ is amenable by assumption, there is $a \in C_c(H)^+$ as in Definition \ref{def:amenability-through-Reiter}, that is $\sum_{\gamma\in H^x}a(\gamma)\leq1\quad(x\in H^{(0)})$, and
\begin{align*}
    \sup_{x\in J}\left|1-\sum_{\gamma\in H^x}a(\gamma)\right|
 <\varepsilon \quad \text{ and } \quad  \sup_{\alpha\in K}\sum_{\gamma\in H^{s(\alpha)}} |a(\alpha\gamma)-a(\gamma)|<\varepsilon.
\end{align*}

We extend $a$ by zero to outside of $H$ to obtain a function $\widetilde a \in C_c(G)^+$. If $x \in G^{(0)}$, then
\begin{align*}
    \sum_{\gamma\in G^x}\widetilde a(\gamma) = \begin{cases}
        \sum_{\gamma\in H^x}a(\gamma), & \text{if } x \in H^{(0)} \\
        0 & \text{if } x \in G^{(0)} \setminus H^{(0)} .
    \end{cases}
\end{align*}
    Thus we have
    \begin{align*}
       \sup_{x \in G^{(0)}} \sum_{\gamma\in G^x}\widetilde a(\gamma)\leq 1 \qquad \text{ and }
\qquad
\sup_{x\in J}\left|1-\sum_{\gamma\in G^x} \widetilde a(\gamma)\right|
 <\varepsilon.
    \end{align*}
    If $\alpha \in K \subseteq H$ and $\gamma \in G^{s(\alpha)}$, then $\gamma \in H$ if and only if $\alpha \gamma \in H$ as $H$ is a subgroupoid of $G$. Consequently
    \begin{align*}
        \sum_{\gamma\in G^{s(\alpha)}} |\widetilde a(\alpha\gamma)-\widetilde a(\gamma)|
 =\sum_{\gamma\in H^{s(\alpha)}} |a(\alpha\gamma)-a(\gamma)|.
    \end{align*}
Thus $\widetilde a$ also satisfies the required invariance statement from Definition \ref{def:amenability-through-Reiter}. 

Constructing functions $\widetilde a = \widetilde a_{K,J, \varepsilon}$ as above and indexing them by the directed set determined by $(K,J, \varepsilon)$ for increasing $K, J$ and decreasing $\varepsilon$, we obtain a Reiter net for $G$. It follows that $G$ is amenable. 
    
\end{proof}

It will further simplify the proof of Theorem \ref{thm:subexp-growth-implies-amenable} if we restrict to groupoids with compact unit space. Recall that for an \'etale groupoid $G$, there is an associated groupoid $G^+$ known as the Alexandrov groupoid \cite[Definition 7.8]{KKLRU2021}. For a locally compact Hausdorff \'etale groupoid $G$, let $(G^0)^+=G^0\cup \{\infty\}$ be the one-point compactification of the unit space.  The Alexandrov groupoid $G^+$ is the set $G \cup (G^0)^+$, with topology determined by requiring that the inclusions $G, (G^0)^+ \subseteq G^+$ are open, and the groupoid structure extends that of $G$ by making $\infty$ a unit. In particular, since $G$ is locally compact Hausdorff \'etale, so is $G^+$. 
When $G^{(0)}$ is already compact, the added point is isolated in the one-point compactification. If $\tilde \ell$ is a uniformly fibrewise coarse length function on $G$, we extend it by $\widetilde \ell(\infty) = 0$. The extension is still controlled, because the non-unit part of every compact subset of $G^+$ is compact in $G$. The growth function is then given by $\max \{1, \gamma_{\widetilde \ell } \}$, so the length function has the same growth rate. 

\begin{lemma}\label{lemma:gpd-amenable-iff-alexandrov-amenable}
    Let $G$ be a locally compact Hausdorff \'etale groupoid, and let $G^+$ be its associated Alexandrov groupoid. Then $G$ is topologically amenable if and only if $G^+$ is topologically amenable.
\end{lemma}

\begin{proof}
   If $G$ is amenable, there exists a net $(a_i)_i$ in $C_c(G)^+$ satisfying Reiter's condition. 
    Define a function $d_i:(G^0)^+\to \mathbb C$ by 
    $$d_i(x):=\begin{cases}
        1-\lambda^x(a_i), & x\in G^0\\
        1, & x=\infty.
    \end{cases}$$
    Then $d_i\geq 0$ by assumption on $a_i$. Since the function $x\mapsto 1-\lambda^x(a_i)$ is continuous and constantly equal to $1$ outside of the compact set $r(\operatorname{supp}(a_i))$, the function $d_i$ is continuous. 
    Extending it by zero outside of $(G^0)^+$ we may view $d_i\in C_c(G^+)$. Similarly, since $a_i\in C_c(G)$ we may extend it by zero and regard $a_i$ as a function in $C_c(G^+)$. 
    Let $b_i:=a_i+d_i$. We claim that the net $(b_i)_i$ is a Reiter net for $G^+$ as in Definition \ref{def:amenability-through-Reiter}. The unit space compactification is clopen in $G^+$, so $b_i$ is continuous, positive, and its support is compact. 
    We have
    $$
    \lambda^x(b_i)=\lambda^x(a_i)+\lambda^x(d_i)=1 \text{ for all }x\in (G^0)^+,
    $$
    and hence the net $(b_i)_i$ satisfies items (1) and (2) in Definition \ref{def:amenability-through-Reiter}. 
    
    It remains to verify (3). Let $K\subseteq G^+$ be compact. Since $\sum_{\gamma\in (G^+)^{s(\alpha)}} \vert b_i(\alpha\gamma)-b_i(\gamma)\vert =0$ for $\alpha\in K\cap (G^0)^+$ it is enough to consider $\alpha$ ranging over the set $L=K\setminus (G^0)^+$. 
    Note that $L$ is a compact subset of $G$. For $\alpha\in L$ and $\gamma\in G^{s(\alpha)}$ we have
    $$
    b_i(\alpha\gamma)-b_i(\gamma)=a_i(\alpha\gamma)-a_i(\gamma)+d_i(\alpha\gamma)-d_i(\gamma).
    $$
    Since $d_i$ is supported on the unit space,
\[
\sum_{\gamma\in G^{s(\alpha)}}
 |b_i(\alpha\gamma)-b_i(\gamma)|
\leq
\sum_{\gamma\in G^{s(\alpha)}}|a_i(\alpha\gamma)-a_i(\gamma)|
 +d_i(r(\alpha))+d_i(s(\alpha)).
\]
Taking suprema over $L$ and using Reiter's conditions (2) and (3),
\[
\begin{split}
\sup_{\alpha\in K}\sum_{\gamma\in(G^+)^{s(\alpha)}}
 |b_i(\alpha\gamma)-b_i(\gamma)|
&\leq \sup_{\alpha\in L}\sum_{\gamma\in G^{s(\alpha)}}
 |a_i(\alpha\gamma)-a_i(\gamma)|\\
&\quad+\sup_{x\in r(L)}(1-\lambda^x(a_i))
 +\sup_{x\in s(L)}(1-\lambda^x(a_i))\longrightarrow 0.
\end{split}
\]

    For the converse, observe that $G^+$ contains $G$ as an open subgroupoid. Hence the result follows from \cite[Proposition~5.1.1]{DelarocheRenaultAmenableGpds2000}.
\end{proof}

The following result allows us to replace any given controlled length function by a continuous proper length function which inherits its growth properties. The proof is closely modelled on \cite[Theorem~4.10]{MaWu2020}.

\begin{lemma}\label{lemma:cont-length-same-growth}
    Let $G$ be a $\sigma$-compact locally compact Hausdorff \'etale groupoid with compact unit space. For every controlled length function $\widetilde \ell$ on $G$, there is a continuous proper length function $\ell$ on $G$ such that $\widetilde \ell \leq \ell$. Consequently, subexponential growth of $\widetilde \ell$ implies subexponential growth of $\ell$.
\end{lemma}

\begin{proof}
    To ease notation, we set $X := G^{(0)}$ and $Y = G \setminus X$. Since $G$ is Hausdorff and \'etale, $X$ is clopen in $G$, and so $Y$ is clopen, locally compact Hausdorff and $\sigma$-compact. 

    If $Y = \emptyset$, we take $\ell = 0$, which is a proper length function because $X$ is compact, and so we are done. We therefore assume $Y \neq \emptyset$ for the rest of the proof. By \cite[Lemma~4.9]{MaWu2020}, there exists a continuous proper function $q:Y\to [0,\infty)$.

    As $\widetilde \ell$ is controlled, the numbers
    \begin{align*}
        M_n=\sup\bigl(\{\widetilde\ell(y):y\in Y,\ q(y)\leq n+2\}
                 \cup\{0\}\bigr),\qquad n\geq0,
    \end{align*}
    are finite and nondecreasing. Define $b \colon [0,\infty) \to \mathbb R$ by $b(n) = n + 2 + M_n$ at the integers, and extend $b$ by linear interpolation between consecutive integers. Then $b$ is continuous and nondecreasing. 
    For $n\leq t \leq n+1$, we have $b(t) \geq n +2+M_n \geq t+1$, and $q(y) \leq t+1$ implies $\widetilde\ell (y) \leq M_n \leq b(t)$. Define 
    \begin{align*}
        \rho(g) = \begin{cases}
            0 & \text{if } g\in X \\
            \max\{ b(q(g)), b(q(g^{-1})) \} & \text{if } g\in Y. 
        \end{cases}
    \end{align*}
      As $X$ and $Y$ are clopen, the function $\rho$ is continuous. We also observe that it is symmetric, and satisfies $\rho \geq \widetilde \ell$ and $\rho \geq 1$ on $Y$. 
      For every $R\geq0$, the part of $\rho^{-1}([0,R])$ in $Y$ is closed in a compact sublevel set of $q$ (and is empty if $R<1$). Together with compactness of $X$, this shows that every sublevel set of $\rho$ is compact.

    We proceed to define the continuous length function $\ell$. For $n\geq1$, let $G^{(n)}$ denote the composable $n$-tuples of arrows. Put
    \begin{align*}
        \ell(g)=\inf\left\{\sum_{j=1}^n\rho(g_j):  n\geq1,\ (g_1,\ldots,g_n)\in G^{(n)},\ g_1\cdots g_n=g\right\}.
    \end{align*}
    First, observe that the one-factor decomposition shows that $0\leq\ell(g)\leq\rho(g)$, so $\ell$ is finite, and $\ell=0$ on $X$. By reversing and inverting a decomposition, and using symmetry of $\rho$, we observe $\ell(g^{-1})=\ell(g)$. 
     If $(g,h)$ is composable, we concatenate decompositions of $g$ and $h$ with costs less than $\ell(g)+\varepsilon/2$ and $\ell(h)+\varepsilon/2$. This gives $\ell(gh)\leq\ell(g)+\ell(h)+\varepsilon$. Then letting $\varepsilon\downarrow0$ proves subadditivity. It follows that $\ell$ is a length function. 
     Furthermore, we see that every decomposition of $g$ satisfies
     \begin{align*}
         \widetilde\ell(g)\leq\sum_{j=1}^n\widetilde\ell(g_j) \leq\sum_{j=1}^n\rho(g_j),
     \end{align*}
     from which it follows that $\widetilde \ell(g) \leq \ell (g)$ for all $g\in G$.

    The infimum defining $\ell$ is in fact attained. To see this, observe first that if $g \in X$, then the one-factor decomposition $g$ has cost zero. For $g \in Y$, we choose for each integer $j \geq 1$ a decomposition with cost less than $\ell(g) + 1/j$, and remove all unit factors. Note that there is at least one remaining factor, and every remaining factor has a cost of at least $1$. Thus the number of factors is bounded by $\lfloor\ell(g)+1\rfloor$. 
    On a subsequence the number of factors is a fixed integer $n$. Every factor belongs to the compact set $Q=Y\cap\rho^{-1}([0,\ell(g)+1])$.
    The resulting tuples have a convergent subnet in $Q^n$. We observe that composability is a closed condition as it is expressed through equality of the range and source maps into the Hausdorff space $X$. As the range and source maps are continuous, this condition passes through limits. Continuity of multiplication therefore implies that the limiting tuple is a decomposition of $g$. Continuity of $\rho$ shows that its cost must be $\ell(g)$. Note in particular that $\ell(g) \geq 1$ for every $g\in G\setminus X$. 

    For $R\geq0$, attainment and the same bound on the number of factors give 
    \begin{align*}
        B_\ell(R)=X\ \cup \bigcup_{1\leq n\leq\lfloor R\rfloor} \left\{g_1\cdots g_n:
 (g_1,\ldots,g_n)\in G^{(n)}\cap Y^n, \sum_{j=1}^n\rho(g_j)\leq R\right\}.
    \end{align*}
    For each $n$, the tuples in the union form a closed subset of the compact space $(Y\cap\rho^{-1}([0,R]))^n$, and their image under the continuous operation of multiplication is compact, hence closed.
    The union is finite, and is empty if $R < 1$. Thus every $B_\ell(R)$ is compact, and hence closed in $G$. It follows that $\ell$ is lower semicontinuous. 

    To prove upper semicontinuity, first note that $\ell = 0$ on the open set $X$. For $g \in Y$ and $\varepsilon > 0$, take a length attaining decomposition $g = g_1 \cdots g_n$. Then, choose open bisections $W_j$ containing $g_j$. Using continuity of $\rho$, shrink them so that $\rho(h) < \rho(g_j) + \varepsilon/n$ for every $h \in W_j$. The product $W_1 \cdots W_n$ is an open neighbourhood of $g$. By construction, every $h$ in this neighbourhood has a decomposition $h = h_1 \cdots h_n$ with $h_j \in W_j$, and so
    \begin{align*}
        \ell(h)\leq\sum_{j=1}^n\rho(h_j) <\sum_{j=1}^n\rho(g_j)+\varepsilon =\ell(g)+\varepsilon.
    \end{align*}
    This proves upper semicontinuity, and hence continuity of $\ell$. The fact that $B_\ell(R)$ is compact for all $R \geq 0$ then implies that $\ell$ is a proper map. 
    The continuous proper length $\ell$ is uniformly fibrewise coarse as it is bounded on compact sets, and each compact ball is covered by finitely many open bisections, each meeting a source fibre in at most one arrow.
    Moreover, $B_\ell(R)\subseteq B_{\widetilde\ell}(R)$ for all $R \geq 0$ implies $\gamma_\ell(R)\leq\gamma_{\widetilde\ell}(R)$ for all $R \geq 0$.
    As such, if $\widetilde \ell$ has subexponential growth, so does $\ell$. 
    \end{proof}

We may now prove the first main result of the article.

\begin{theorem}\label{thm:subexp-growth-implies-amenable}
    Let $G$ be a locally compact Hausdorff \'{e}tale groupoid. Suppose that
$G$ admits a uniformly fibrewise coarse length function $\widetilde \ell:G\longrightarrow [0,\infty)$ such that $\gamma_{\widetilde\ell}$ has subexponential growth. Then $G$ is topologically amenable.
\end{theorem}

\begin{proof}
    By Lemma \ref{lemma:sigma-cpct-amen-suffices} it suffices to prove amenability of an open $\sigma$-compact subgroupoid $H$ of $G$. The restriction of $\widetilde \ell$ to $H$ is controlled and has at most the same growth. We pass to the Alexandrov groupoid $H^+$ and extend the length by zero in the added unit as above. As noted, the extension is controlled and has the same growth type. Then Lemma \ref{lemma:cont-length-same-growth} provides a continuous proper length $\ell \geq \widetilde \ell$ on $H^+$, from which we deduce $\gamma_\ell \leq \gamma_{\widetilde \ell}$, and therefore $\gamma_{\ell}$ also has subexponential growth. For the remainder of the proof, we write $G$ instead of $H^+$. 

    The proof will proceed by constructing a net as required by Definition \ref{def:amenability-through-Reiter} explicitly. 
    For $t>0$ let $f_t:G\to \mathbb C$ be the function $f_t(\gamma)=e^{-t\ell(\gamma)}$. Since $\ell$ is continuous and proper, $f_t\in C_0(G)^+$. 
We claim that $\sup_{x\in G^0} \lambda^x(f_t)<\infty$.
Inversion identifies $B_\ell(R)\cap G^x$ with $B_\ell(R)x$, since $\ell$ is symmetric. 
For $n\in\mathbb N_0$ consider the annulus
\[
        A_x(n):=\{\gamma\in G^x:n\leq \ell(\gamma)<n+1\}.
\]
Since 
$A_x(n)\subseteq B_\ell(n+1) \cap G^x$ 
we have $|A_x(n)|\leq \gamma_{\ell}(n+1).$
Hence
\[        \lambda^x(f_t)=\sum_{\gamma\in G^x} e^{-t \ell(\gamma)}                  = \sum_{n=0}^\infty \sum_{\gamma\in A_x(n)} e^{-t\ell(\gamma)}           
        \leq \sum_{n=0}^{\infty} e^{-tn}|A_x(n)|                    
        \leq \sum_{n=0}^{\infty} e^{-tn}\gamma_\ell(n+1).
\]
We claim that the latter series converges. To see this fix $0<a<t$. Since $\gamma_\ell$ has subexponential growth there exists a constant $C_a$ such that  $\gamma_\ell(n+1)\leq C_a e^{a(n+1)}$ for all $n\in \mathbb N_0$. Thus, we have $e^{-tn}\gamma_\ell(n+1)\leq C_a e^a e^{(a-t)n}$ for all $n\in \mathbb N_0$. The comparison theorem gives the desired conclusion.

Note that the upper bound does not depend on $x$ and hence we get $\sup_{x\in G^{(0)}} \lambda^x(f_t)<\infty$ as desired.
Moreover $\lambda^x(f_t)\geq 1$, since $x\in G^x$ and $\ell(x)=0$. Hence
$
        1\leq \lambda^x(f_t)<\infty
$
for all $x\in G^{(0)}$.
Define the fibre-normalised exponential density
\[
        c_t(\gamma):=
        \frac{f_t(\gamma)}{\lambda^{r(\gamma)}(f_t)}.
\]
By construction we have $\lambda^x(c_t)=1$ for all $x\in G^{(0)}$.
We will now show that $c_t$ is approximately invariant in the sense of item (3) in Reiter's condition. To see this consider a composable pair $(\alpha,\eta)\in G^{(2)}$. Since $\eta=\alpha^{-1}(\alpha\eta)$ we obtain
\[
        \ell(\eta)
        \leq \ell(\alpha^{-1})+\ell(\alpha\eta)
        =\ell(\alpha)+\ell(\alpha\eta),
\]
or equivalently
$$
\ell(\alpha\eta)\geq \ell(\eta)-\ell(\alpha).
$$
This length estimate together with the standard estimate $\ell(\alpha\eta)\leq \ell(\alpha)+\ell(\eta)$ gives
\[
        e^{-t \ell(\alpha)}e^{-t \ell(\eta)}
        \leq e^{-t \ell(\alpha\eta)}
        \leq e^{t \ell(\alpha)}e^{-t \ell(\eta)}.
\]
Using Lemma \ref{Lemma: invariance} 
together with this inequality we get
$$\lambda^{r(\alpha)}(f_t)=\sum_{\eta\in G^{r(\alpha)}}f_t(\eta)=\sum_{\eta\in G^{s(\alpha)}} f_t(\alpha\eta)\leq e^{t\ell(\alpha)} \sum_{\eta\in G^{s(\alpha)}} f_t(\eta),$$
and similarly
$$\lambda^{r(\alpha)}(f_t)\geq e^{-t\ell(\alpha)}\lambda^{s(\alpha)}(f_t).$$
Using these estimates we can compare $c_t(\alpha\eta)$ with $c_t(\eta)$: we have
\[
        e^{-2t\ell(\alpha)}
        \leq
        \frac{c_t(\alpha\eta)}{c_t(\eta)}
        \leq
        e^{2t\ell(\alpha)},
\]
and consequently,
\[
        |c_t(\alpha\eta)-c_t(\eta)|
        \leq
        \bigl(e^{2t\ell(\alpha)}-1\bigr)c_t(\eta).
\]
Summing over $\eta\in G^{s(\alpha)}$ gives
\[
        \sum_{\eta\in G^{s(\alpha)}}
        |c_t(\alpha\eta)-c_t(\eta)|
        \leq
        e^{2t \ell(\alpha)}-1.
\]
For a compact subset $K\subseteq G$, let $M_K:=\sup\{\ell(\alpha)\mid \alpha\in K\}$ which is finite since $\ell$ is continuous. Hence
\[
        \sup_{\alpha\in K}
        \sum_{\eta\in G^{s(\alpha)}}
        |c_t(\alpha\eta)-c_t(\eta)|
        \leq
        e^{2tM_K}-1 \to 0 \qquad \text{as }t\to 0.
\]

Thus, $(c_t)_{t>0}$ satisfies Reiter's conditions (1), (2), and (3). Its main deficiency is that the functions $c_t$ are not compactly supported.\footnote{Technically, we also do not know at this point that $c_t$ is continuous. Although it will follow from later arguments, it is irrelevant for the rest of the proof and hence we shall not dwell on this point.}
We now replace $(c_t)_t$ by compactly supported continuous functions. For
$R>0$, consider the continuous, nonincreasing function $\chi_R:[0,\infty)\longrightarrow [0,1]$ given by
$$
\chi_R(t)=\begin{cases}
    1, & \text{for } 0\leq t\leq R,\\
    R+1-t, & \text{for } R< t\leq R+1,\\
    0, & \text{for } t> R+1,
\end{cases}
$$
and define
\[
        f_{t,R}(\gamma):=\chi_R(\ell(\gamma))f_t(\gamma).
\]
Since $\ell$ is continuous and proper, $f_{t,R}\in C_c(G)^+$.
Because $G$ is \'{e}tale and $f_{t,R}\in C_c(G)$, the fibre-sum function $\lambda(f_{t,R})(x)=\lambda^x(f_{t,R})$ is continuous on $G^{(0)}$. Note that we also have $\lambda^x(f_{t,R})\geq f_{t,R}(x)=\chi_R(\ell(x))f_t(x)=\chi_R(0)=1$ for all $x\in G^{(0)}$.

Define
\[
        a_{t,R}(\gamma):=
        \frac{f_{t,R}(\gamma)}{\lambda^{r(\gamma)}(f_{t,R})}.
\]
Then $a_{t,R}\in C_c(G)^+$, and for every $x\in G^{(0)}$,
\[
\sum_{\gamma\in G^x} a_{t,R}(\gamma)
        =
        \frac{1}{\lambda^x(f_{t,R})}
        \sum_{\gamma\in G^x} f_{t,R}(\gamma)
        =1.
\]
Consequently, the collection of functions $(a_{t,R})_{(t,R)\in (0,1]\times (0,\infty)}$ satisfies items (1) and (2) from Reiter's condition. 
It remains to ensure approximate invariance. 
To see this we will first show that $a_{t,R}$ approximates $c_t$ uniformly on fibres. Fix
$t>0$ and define
\[
        T_R(t):=
        \sup_{x\in G^{(0)}}
        \sum_{\substack{\gamma\in G^x\\ \ell(\gamma)>R}}
        e^{-t\ell(\gamma)}.
\]

We claim that $\lim_{R\to \infty} T_R(t)= 0$. This argument uses the same annulus decomposition of the fibres as above:
Indeed, if $N\leq R<N+1$, then
\[
        T_R(t)
        \leq
        \sup_{x\in G^{(0)}}
        \sum_{n\geq N}
        \sum_{\substack{\gamma\in G^x\\ n\leq \ell(\gamma)<n+1}}
        e^{-t\ell(\gamma)}                                                        
        \leq
        \sum_{n\geq N} e^{-tn}\gamma_\ell(n+1),
\]
and the latter is a tail of a convergent series.
Moreover, we have 
\[
        0\leq \lambda^x(f_t)-\lambda^x(f_{t,R})
        =
        \sum_{\gamma\in G^x}
        \bigl(1-\chi_R(\ell(\gamma))\bigr)e^{-t\ell(\gamma)}
        \leq T_R(t) \qquad \text{for all }x\in G^{(0)}.
\] 
Hence
\[
\begin{aligned}
\sum_{\gamma\in G^x} \vert a_{t,R}(\gamma)-c_t(\gamma)\vert & = \sum_{\gamma\in G^x}
\left|
\frac{\chi_R(\ell(\gamma))e^{-t\ell(\gamma)}}{\lambda^x(f_{t,R})}
-
\frac{e^{-t\ell(\gamma)}}{\lambda^x(f_t)}
\right|                                                        \\
&\leq
\sum_{\gamma\in G^x}
\chi_R(\ell(\gamma))e^{-t\ell(\gamma)}
\left|
\frac{1}{\lambda^x(f_{t,R})}-\frac{1}{\lambda^x(f_t)}
\right|                                                        \\
&\quad+
\sum_{\gamma\in G^x}
\frac{(1-\chi_R(\ell(\gamma)))e^{-t\ell(\gamma)}}{\lambda^x(f_t)}             \\
&=
\lambda^x(f_{t,R})
\left|
\frac{1}{\lambda^x(f_{t,R})}-\frac{1}{\lambda^x(f_t)}
\right|
+
\frac{\lambda^x(f_t)-\lambda^x(f_{t,R})}{\lambda^x(f_t)}                                \\
&=
\lambda^x(f_{t,R})
\left|
\frac{\lambda^x(f_t)-\lambda^x(f_{t,R})}{\lambda^x(f_{t,R})\lambda^x(f_t)}
\right|
+
\frac{\lambda^x(f_t)-\lambda^x(f_{t,R})}{\lambda^x(f_t)}                                \\
&\leq
2T_R(t),
\end{aligned}
\]
where we used the inequality $\lambda^x(f_t)\geq 1$ in the final step.
Therefore
\[
        \sup_{x\in G^{(0)}}
        \sum_{\gamma\in G^x}
        |a_{t,R}(\gamma)-c_t(\gamma)|
        \leq 2T_R(t)\longrightarrow 0, \qquad \text{as }R\to \infty.
\]

Now let $K\subseteq G$ be compact. For every $\alpha\in K$, we have
\[
\begin{aligned}
&\sum_{\eta\in G^{s(\alpha)}}
|a_{t,R}(\alpha\eta)-a_{t,R}(\eta)|                              \\
&\leq
\sum_{\eta\in G^{s(\alpha)}}
|a_{t,R}(\alpha\eta)-c_t(\alpha\eta)|                              \\
&\quad+
\sum_{\eta\in G^{s(\alpha)}}
|c_t(\alpha\eta)-c_t(\eta)|                                        \\
&\quad+
\sum_{\eta\in G^{s(\alpha)}}
|c_t(\eta)-a_{t,R}(\eta)|.
\end{aligned}
\]
Using Lemma \ref{Lemma: invariance} the first term is bounded by
\[
        \sup_{x\in G^{(0)}}
        \sum_{\gamma\in G^x}|a_{t,R}(\gamma)-c_t(\gamma)|.
\]
The third term is bounded in the same way. Hence
\[
\begin{aligned}
&\sup_{\alpha\in K}
\sum_{\eta\in G^{s(\alpha)}}
|a_{t,R}(\alpha\eta)-a_{t,R}(\eta)|                              \\
&\leq
2\sup_{x\in G^{(0)}}
\sum_{\gamma\in G^x}|a_{t,R}(\gamma)-c_t(\gamma)|
+
e^{2tM_K}-1                                                       \\
&\leq
4T_R(t)+e^{2tM_K}-1.
\end{aligned}
\]

As the first summand $4T_R(t)$ is independent of $K$ and the second term $e^{2tM_K}-1$ tends to $0$ for each fixed $K$ as well, we may choose a Reiter sequence:
For each $n\in \mathbb N$ we can use $\lim_{R\to \infty} T_R(\frac{1}{n})=0$ to find an $R_n>0$  large enough that $4T_{R_n}(\frac{1}{n})<\frac{1}{n}$. For $n\in \mathbb N$ put $b_n:=a_{\frac{1}{n},R_n}$. We claim that $(b_n)_n$ is a Reiter sequence for $G$. We already observed that conditions (1) and (2) are satisfied. For condition (3) let $K\subseteq G$ be a compact subset. 
Then by the inequality from above we have 
$$
\sup_{\alpha\in K}
\sum_{\eta\in G^{s(\alpha)}}
|b_n(\alpha\eta)-b_n(\eta)|                              \\
\leq
4T_{R_n}(1/n)+e^{2M_K/n}-1<1/n+e^{2M_K/n}-1\to 0.
$$
This shows that $G$ admits a Reiter sequence, and thus $G$ is topologically amenable. 
Lemma \ref{lemma:gpd-amenable-iff-alexandrov-amenable} then yields amenability of $H$, and Lemma \ref{lemma:sigma-cpct-amen-suffices} then shows that the original groupoid is amenable. 
\end{proof}

As an example application of this result we can recover \cite[Theorem~6.1]{Tu01} using completely different methods:
for a uniformly locally finite metric space $X$, let $G(X)$ denote its coarse groupoid. By \cite[Theorem 5.3]{SkandalisTuYu2002} property (A) for $X$ is equivalent to topological amenability of $G(X)$. 
If the ball cardinalities of $X$ have a uniformly subexponential upper bound $f$, then \cite[Lemma 5.6 and Proposition 5.7]{AOP-JNCG} provide a continuous length function on the coarse groupoid with growth bounded by a constant multiple of $f^2$. This bound is still subexponential. Theorem \ref{thm:subexp-growth-implies-amenable} therefore implies amenability of $G(X)$, and hence property (A) for $X$.

\section{Uniformly bounded doubling and comparison}\label{sec:comparison-and-examples}
In this section we prove the second main result of the article, namely Theorem \ref{thm:main-theorem-wocg}, which tells us that if a locally compact Hausdorff \'etale groupoid $G$ with compact and metrizable unit space admits a uniformly fibrewise  coarse length function with uniformly bounded doubling which is comparison admissible, then $G$ has weak $m$-comparison, where $m$ is determined by the doubling and admissibility constants. We also prove Corollary \ref{cor:AH-conjecture}, which establishes the AH-conjecture for minimal ample groupoids admitting such comparison admissible doubling length functions, provided their unit spaces have no isolated points. 

\begin{definition}\label{def:subequivalence}
    Suppose $G$ is  
    a locally compact Hausdorff 
    \'etale groupoid with compact unit space. If $A, B \subseteq \Go$, we say $A$ is \emph{$m$-subequivalent to} $B$, and write $A \precsim_m B$, if there is a finite collection  $\mathcal{U} = \{ V_j \}_{j=1}^{\vert \mathcal{U}\vert}$ of open bisections, and a partition of $\mathcal{U}$ into subcollections $\mathcal{U}_0, \ldots , \mathcal{U}_m$ such that
    \begin{enumerate}
        \item $A \subseteq \bigcup_{V \in \mathcal{U}} s(V)$, and
        \item for each $i = 0, \ldots ,m$ we have $r(V) \subseteq B$ for all $V \in \mathcal{U}_i$, and $r(V) \cap r(V') = \emptyset$ for all $V \neq V' \in \mathcal{U}_i$. 
    \end{enumerate}
\end{definition}

Given 
a locally compact Hausdorff
\'etale groupoid $G$ with compact unit space and a Borel measure $\mu$ on $\Go$, we say $\mu$ is \emph{$G$-invariant}, or simply \emph{invariant}, if $\mu(s(V)) = \mu(r(V))$ for all open bisections $V \subseteq G$. Denote 
by $M(G)$ the convex set of invariant positive regular Borel probability measures on $\Go$. It is compact in the weak-$*$-topology inherited from the dual space of $C(\Go)$. 

\begin{definition}\label{def:dynamical-comparison}
    Suppose $G$ is  
    a locally compact Hausdorff 
    \'etale groupoid with compact unit space, and suppose $M(G) \neq \emptyset$. 
    We say $G$ has \emph{$m$-comparison} if for any closed $A \subseteq \Go$ and open $B \subseteq \Go$ with $A \subseteq r(GB)$ and for which $\mu(A) < \mu(B)$ for all $\mu \in M(G)$, we have $A \precsim_m B$.  If $m=0$ we simply say $G$ has \emph{comparison}. 
    
    Similarly, we 
    say $G$ has \emph{weak $m$-comparison} if for any closed $A \subseteq \Go$ and open $B \subseteq \Go$ with $A \subseteq r(GB)$ and for which $\sup_{\mu \in M(G)} \mu(A) < \inf_{\mu \in M(G)}\mu(B)$,  we have $A \precsim_m B$.
\end{definition}

\begin{definition}\label{def:lf-admissibility}
     Suppose $G$ is a locally compact Hausdorff \'etale groupoid with compact unit space.
     Let $\ell \colon G\to [0, \infty)$ be a uniformly fibrewise coarse length function. We say $\ell$ is \emph{comparison admissible} 
     if there is $\Lambda \geq 1$ such that for every sufficiently large $R$ there are finitely many compact open bisections $V_{R,1}, \ldots , V_{R, N_R}$ for which $D_R := \cup_{i=1}^{N_R} V_{R,i}$  satisfies $B_\ell(R) \subseteq D_R \subseteq B_\ell(\Lambda R)$. 
\end{definition}
Suppose $\ell \colon G \to [0,\infty)$ is as in Definition \ref{def:lf-admissibility},  suppose a sufficiently large $R \geq 0$ is given, and suppose $V_{R,1}, \ldots , V_{R, N_R}$ witness the comparison admissibility. We note that $s(V_{R,i})$ and $r(V_{R,i})$ are open for all $i=1, \ldots , N_R$. Moreover, by adjoining the inverses $\{ V_{R,i}^{-1} \}_{i=1}^{N_R}$ if necessary, we may assume $D_R^{-1} = D_R$. We will in fact tacitly do so in the sequel.

Now, for $Y \subseteq G^{(0)}$, $x \in G^{(0)}$ and $R > 0$, define
\begin{align}
    N_R(Y,x) = \sum_{\gamma\in B_\ell(R)x}
 (R-\ell(\gamma)) 1_Y(r(\gamma))
\end{align}
and 
\begin{align}
    V_R(x) = N_R(G^{(0)}, x).
\end{align}
Note that $V_R(x) \geq R$ for all $R \geq 0$, because $x \in B_\ell(t)x$ for all $t \geq 0$ and $\ell(x) = 0$. We now define the upper and lower Ces\`aro densities
\begin{align}\label{eq:Cesaro-densities}
    \overline{D}_\ell(Y) &= \limsup_{R \to \infty} \sup_{x \in G^{(0)}} \frac{N_R(Y,x)}{V_R(x) }\notag\\
    \underline{D}_\ell(Y) &= \liminf_{R \to \infty} \inf_{x \in G^{(0)}} \frac{N_R(Y,x)}{V_R(x) }.
\end{align}

\begin{remark}\label{remark:comparison-admissible-implies-sigma-compact}
    Note that if $\ell \colon G \to [0,\infty)$ is a comparison admissible length function, then for sufficiently large radii $R_n \to \infty$ and corresponding compact sets $D_{R_n}$, we have $G = \cup_n D_{R_n}$, showing that $G$ must be $\sigma$-compact. If the unit space is metrizable, it is second-countable. A countable cover of $G$ by open bisections then shows that $G$ is second-countable as well. 

    Moreover, any transformation groupoid $\Gamma \ltimes X$ with $\Gamma$ finitely generated and $X$ compact admits a comparison admissible length function. Indeed, let $\ell_\Gamma$ be a word length function on $\Gamma$, and set $\ell(\gamma, x) = \ell_\Gamma (\gamma)$. Then 
    \begin{align*}
        B_\ell(n) = \bigcup_{\gamma \in B_{\ell_\Gamma}(n)} \{\gamma \} \times X
    \end{align*}
    showing that $\ell$ is comparison admissible with $\Lambda = 1$. 
\end{remark}

\begin{lemma}\label{lemma:density-approx}
    Let $G$ be a locally compact Hausdorff \'etale groupoid with compact
metrizable unit space. Suppose that $\ell$ is a uniformly fibrewise coarse
length function and, for some $C_{\mathrm{dbl}}\geq1$
\begin{align*}
    |B_\ell(2t)x|\leq C_{\mathrm{dbl}}|B_\ell(t)x|
\end{align*}
for all $t > 0$ and all $x \in G^{(0)}$. Then $M(G) \neq \emptyset$, and the following hold.
\begin{enumerate}
    \item For every closed $A \subseteq G^{(0)}$ we have
    \begin{align*}
        \overline{D}_\ell(A) \leq\sup_{\mu\in M(G)}\mu(A).
    \end{align*}
    \item For every open $A \subseteq G^{(0)}$ we have
    \begin{align*}
        \inf_{\mu\in M(G)}\mu(A)\leq\underline{D}_\ell(A).
    \end{align*}
\end{enumerate}
\end{lemma}
\begin{proof}
   First, for $R >0$ and $x \in G^{(0)}$ we define the state $L_{R,x}$ on $C(G^{(0)})$ by
   \begin{align*}
       L_{R,x}(f) := 
       \frac{1}{V_R(x)}  \sum_{\gamma\in B_\ell(R)x} (R - \ell(\gamma)) \cdot f(r(\gamma)).
   \end{align*}
   We fix a closed set $A \subseteq G^{(0)}$. Pick a sequence $R_k \to \infty$ and $x_k \in G^{(0)}$ such that
   \begin{align*}
        \lim_{k\to \infty}\frac{ N_{R_k}(A,x_k)}{ V_{R_k}(x_k)} = \overline{D}_\ell(A),
   \end{align*}
    set $L_k := L_{R_k, x_k}$, and let $\nu_k$ be the probability measure on $G^{(0)}$ corresponding to it.
    Since $G^{(0)}$ is compact and metrizable, its space of probability measures is weak$^*$ compact and metrizable, so we may pass to a subsequence of $(L_k)_k$ which converges to a state $L$ represented by a probability measure $\nu$ on $G^{(0)}$. We claim that $\nu$ is invariant.

    To see this, let $U \subseteq G$ be a precompact open bisection, and denote by $\theta_U \colon s(U) \to r(U)$ the associated partial homeomorphism. Since $\ell$ is uniformly fibrewise coarse, it is in particular controlled, and so $K_U := \sup_{\gamma \in U} \ell(\gamma) < \infty$.
    For any $f \in C_c(s(U))$, also viewed as extended by zero to $G^{(0)}$, let $\check{f}$ denote the function $f \circ \theta_U^{-1}$ extended by zero outside $r(U)$. Denote by $w_R$ the function
    \begin{align*}
        w_R(t) = \begin{cases}
            R-t & \text{if $R-t \geq 0$} \\
            0 & \text{if $R-t < 0$},
        \end{cases}
    \end{align*}
    and note that the function is Lipschitz continuous with Lipschitz constant $1$. 
    For $y \in s(U)$, denote by $u_y$ the unique element of $U$ with $s(u_y) = y$. The map $\gamma \mapsto  u_{r(\gamma)} \gamma $ yields a bijection 
    \begin{align*}
        \{ \gamma \in G_x \mid r(\gamma) \in s(U) \} \to \{ \mu \in G_x \mid r(\mu) \in r(U) \}.
    \end{align*}
We may then calculate
\begin{align*}
    V_R(x) \vert L_{R,x}(\check{f}) &- L_{R,x}(f) \vert \\
    &= \left\vert \sum_{\gamma \in B_\ell(R)x} (R - \ell(\gamma)) \cdot \check{f}(r(\gamma)) - \sum_{\mu \in B_\ell(R)x} (R- \ell(\mu)) f(r(\mu)) \right\vert \\
    &= \left\vert\sum_{\{\gamma \in G_x \mid r(\gamma) \in r(U)\}} w_R(\ell(\gamma)) \check{f}(r(\gamma)) - \sum_{\{\mu \in G_x \mid r(\mu) \in s(U)\}} w_R(\ell(\mu)) f(r(\mu))\right\vert \\
    &= \left\vert\sum_{\{\gamma \in G_x \mid r(\gamma) \in s(U)\}} (w_R(\ell(u_{r(\gamma)} \gamma)) - w_R(\ell(\gamma))) f(r(\gamma))  \right\vert \\
    &\leq \sum_{\{\gamma \in G_x \mid r(\gamma) \in s(U)\}} \vert w_R(\ell(u_{r(\gamma)} \gamma)) - w_R(\ell(\gamma))\vert \vert f(r(\gamma)) \vert .
\end{align*}
    Using that $w_R$ is $1$-Lipschitz continuous together with
    \begin{align*}
        \vert \ell(u_{r(\gamma)} \gamma ) - \ell(\gamma) \vert \leq \ell(u_{r(\gamma)}) \leq \sup_{\mu \in U} \ell(\mu) = K_U,
    \end{align*}
    and that $w_R(\ell(u_{r(\gamma)}\gamma)) = w_R(\ell(\gamma)) = 0$ on $B_\ell(R + K_U)^{c}$, we deduce
    \begin{align*}
        V_R(x) \vert L_{R,x}(\check{f}) - L_{R,x}(f) \vert &\leq K_U \Vert f \Vert_\infty \cdot \vert B_\ell(R + K_U)x \vert,
    \end{align*}
    or in other words,
    \begin{align*}
        \vert L_{R,x}(\check{f}) - L_{R,x}(f) \vert \leq \frac{K_U \Vert f \Vert_\infty \cdot \vert B_\ell(R + K_U) x\vert}{V_R(x)}.
    \end{align*}
    Now we note the straightforward observation 
    \begin{align*}
        V_R(x) = \sum_{\gamma \in B_\ell(R)x} (R - \ell(\gamma)) \geq \frac{R}{2} \vert B_\ell(\frac{R}{2})x\vert.
    \end{align*}
    Choosing $R \geq K_U$ and using the doubling condition twice we calculate
    \begin{align*}
        \vert L_{R,x}(\check{f}) - L_{R,x}(f) \vert &\leq \frac{K_U \Vert f \Vert_\infty \cdot \vert B_\ell(R + K_U) x \vert}{V_R(x)} \\
        &\leq \frac{K_U \vert B_\ell(2R) x\vert \Vert f\Vert_\infty}{\frac{R}{2} \vert B_\ell(\frac{R}{2})x \vert} \\
        &\leq \frac{2C_{\rm dbl}^2 K_U}{R} \Vert f \Vert_\infty,
    \end{align*}
    where $C_{\rm dbl}$ is the doubling constant of $\ell$. 
    By passing to the limit and noting that the bisection $U$ and $f \in C_c(s(U))$ were arbitrary, we deduce that the measure is invariant, that is, $\nu \in M(G)$. In particular, $M(G) \neq \emptyset$. 

    Now, by definition of $\nu_k$ we have
    \begin{align*}
        \nu_k (A) = \frac{N_{R_k}(A,x_k)}{V_{R_k}(x_k)},
    \end{align*}
    from which the portmanteau theorem yields
    \begin{align*}
        \overline{D}_\ell(A) = \limsup_k \nu_k(A) \leq \nu(A) \leq \sup_{\mu \in M(G)} \mu(A).
    \end{align*}
    This finishes the proof of (1). For (2), we note that if $A$ is open, then
    $N_R(A,x) = V_R(x) - N_R(G^{(0)}\setminus A, x)$, and thus
    \begin{align*}
        \frac{N_R(A,x)}{V_R(x)} = 1- \frac{N_R(G^{(0)} \setminus A,x)}{V_R(x)}.
    \end{align*}
    We then apply (1) to $G^{(0)}\setminus A$ to obtain (2). 
\end{proof}

To prove Lemma \ref{Lemma: m-comparison} below, we need to approximate subsets of $G^{(0)}$ by slightly thicker sets from above, or thinner sets from below, in a controllable manner. We therefore introduce the following notation. Let $G^{(0)}$ be metrizable, with a compatible metric $d$. For $A \subseteq G^{(0)}$ and $\varepsilon>0$, write
\begin{align*}
    A^\varepsilon := \{ x \in G^{(0)} \mid d(x,A) < \varepsilon \}, \quad {A^{[\varepsilon]}} = \{x \in G^{(0)} \mid d(x,A) \leq \varepsilon\}.
\end{align*}
Furthermore, we set
\begin{align*}
    A^{-\varepsilon} := \{ x \in G^{(0)} \mid d(x, G^{(0)} \setminus A) > \varepsilon \} .
\end{align*}
We use $d(x, \emptyset) = \infty$. The notation $A^{[\varepsilon]}$ denotes the closed metric thickening, which need not equal $\overline{A^\varepsilon}$. 
We then make the following observation. 

    \begin{lemma}\label{Lemma:an epsilon of room}
        Let $G$ be a locally compact Hausdorff \'etale groupoid, and suppose $\Go$ is compact and metrizable, and equipped with a compatible metric $d$. 
        Suppose $M(G) \neq \emptyset$. 
        Then the following hold:
        \begin{enumerate}
            \item if $A\subseteq \Go$ is a closed subset, then for every $\varepsilon >0$ there exists a $\delta>0$ such that $\sup_{\mu\in M(G)} \mu({A^{[\delta]}}) \leq \sup_{\mu\in M(G)} \mu(A) + \varepsilon$;
            \item if $U\subseteq \Go$ is an open subset, then for every $\varepsilon >0$ there exists $\delta >0$ such that $\inf_{\mu\in M(G)} \mu(U^{-\delta})\geq \inf_{\mu\in M(G)} \mu(U)-\varepsilon$.
        \end{enumerate}
    \end{lemma}
    \begin{proof}
        For the proof of (1), suppose for contradiction that there exists some $\varepsilon >0$ for which there is no suitable choice of $\delta$. In particular, for every $n\in \mathbb N$ we have
        $$\sup \mu({A^{[\frac{1}{n}]}}) > \sup \mu(A)+\varepsilon.$$
        Thus, there exists a sequence $(\mu_n)_n$ in $M(G)$ such that 
        $$\mu_n({A^{[\frac{1}{n}]}})> \sup \mu(A)+\varepsilon/2.$$
        Since $M(G)$ is weak-$*$-compact, there exists a subsequence $(\mu_{n_k})_k$ of $(\mu_n)_n$ which converges to some $\mu\in M(G)$. For integers $k\geq j\geq 1$ we have
        $$\mu_{n_k}({A^{[\frac{1}{n_j}]}})\geq \mu_{n_k}({A^{[\frac{1}{n_k}]}})>\sup \mu(A) +\varepsilon/2.$$
        Hence, for a fixed $j$ the portmanteau theorem yields, since ${A^{[\frac{1}{n_j}]}}$ is closed, that
        $$\mu({A^{[\frac{1}{n_j}]}})\geq \limsup_k \mu_{n_k}({A^{[\frac{1}{n_j}]}})\geq \sup \mu(A)+\varepsilon/2.$$
        Since $A$ coincides with the decreasing intersection of the sets ${A^{[\frac{1}{n_j}]}}$ for $j\in \mathbb N$, we deduce that
        $$\mu(A)=\lim_{j\to \infty} \mu({A^{[\frac{1}{n_j}]}})\geq \sup \mu(A)+\varepsilon/2,$$
        a contradiction.
        
        For the proof of (2), we apply (1) to $G^{(0)} \setminus U$ and use $U^{-\delta} = G^{(0)} \setminus {(G^{(0)} \setminus U)^{[\delta]}}$.
    \end{proof}

The following lemma generalizes \cite[Lemma 3.1]{NaryshkinComparison2022}. The statement and proof are both complicated by the fact that the natural maps $\theta_i$ defined by open bisections are not amenable to the $\varepsilon$-thickening technique employed in \cite{NaryshkinComparison2022}. More precisely, given an open bisection $V \subseteq G$ with associated partial homeomorphism $\theta_V \colon s(V) \to r(V)$, there is a priori no reason why the composition $\theta_V^{-1}(\theta_V(A)^{\varepsilon})$  need be well-defined. Such compositions are crucial to the proof of \cite[Lemma 3.1]{NaryshkinComparison2022}, and we thus need to find a way to circumvent them. The key lies in an added assumption on the set $D$ in the statement of the following lemma, which allows us to deduce the positive distance margins needed in \eqref{eq:theta-containments}.

\begin{lemma}\label{Lemma: m-comparison}
    Let $G$ be a locally compact Hausdorff étale groupoid with compact unit space $G^{(0)}$, and suppose $d$ is a compatible metric on $G^{(0)}$. Set
    \begin{align*}
        D = \bigcup_{i=1}^N V_i 
    \end{align*}
    where every $V_i$ is a compact open bisection, and suppose $G^{(0)} \subseteq D$. Let $A \subseteq G^{(0)}$ be closed and $B \subseteq G^{(0)}$ be open, and assume $A \subseteq r(GB)$. Let $m \in \N$ and suppose there is $\varepsilon > 0$ such that
    \begin{align*}
        \vert \{\gamma \in D^{-1}D x \mid r(\gamma) \in A^\varepsilon \}\vert < m \vert \{\gamma \in Dx \mid r(\gamma) \in B^{-\varepsilon} \} \vert 
    \end{align*}
    for all $x \in G^{(0)}$. Then $A \precsim_{m-1} B$. 
\end{lemma}

\begin{proof}
    For $i=1, \ldots, N$, denote by $\theta_i \colon s(V_i) \to r(V_i)$ the homeomorphism induced by $V_i$. Since the $V_i$ are compact and open, we note that both $s(V_i)$ and $r(V_i)$ are clopen subsets of $G^{(0)}$ for all $i=1, \ldots , N$. Thus in particular, we have
    \begin{align*}
        d(s(V_i), G^{(0)} \setminus s(V_i) ) > 0 , \qquad d(r(V_i), G^{(0)} \setminus r(V_i)) > 0 \qquad \text{for $i=1,\ldots, N$. }
    \end{align*}
    Now fix a bijection $\sigma \colon \{1, \ldots, mN\} \to \{1 ,\ldots , N\} \times \{1, \ldots, m\}$ which we write $\sigma(n) = (i(n), k(n))$. We will proceed by specifying an algorithm which in finitely many steps will implement the subequivalence $A \precsim_{m-1} B$. Specifically, using a total of $Nm$ steps, we will exhibit families $\mathcal{U}_1, \ldots, \mathcal{U}_m$ of open bisections which satisfy the necessary conditions from Definition \ref{def:subequivalence}. 

    We begin by setting $A_1 = A$, $B_{1,k} = B$ for all $k = 1, \ldots, m$, $\varepsilon_1 = \varepsilon/2$. 
    For $x \in G^{(0)}$, we write
    \begin{align*}
        L_n (x) &:= \{ \gamma \in D^{-1}Dx \mid r(\gamma) \in {A_n^{[\varepsilon_n]}} \} , \text{ and }
        R_{n,k}(x) := \{ \gamma \in Dx \mid r(\gamma) \in B^{-\varepsilon_n}_{n,k} \}.
    \end{align*}
    The induction will preserve the inequality
    \begin{align}\label{eq:algo-LHS-vs-RHS}
        \vert L_n(x) \vert < \sum_{k=1}^m \vert R_{n,k}(x)\vert.
    \end{align}
    For $n=1$, \eqref{eq:algo-LHS-vs-RHS} follows from the assumption.

    Suppose now that we have constructed the data for step $n$, for $1 \leq n \leq mN$, that is, we have constructed $A_n$, $B_{n,k}$ for $k=1,\ldots, m$, and specified an $\varepsilon_n$ for which \eqref{eq:algo-LHS-vs-RHS} is satisfied. 
    Set $i = i(n)$ and $k_0 = k(n)$, and put
    \begin{align*}
        F_{n,i,k_0} := \theta_i^{-1} ((G^{(0)} \setminus B_{n,k_0}) \cap r(V_i)).
    \end{align*}
    Note that the sets $A_n \cap s(V_i)$ and $F_{n,i,k_0}$ are all compact. Both $\theta_i$ and $\theta_i^{-1}$ are uniformly continuous on their compact domains, and so we may choose $0 < \varepsilon_{n+1} < \varepsilon_n$ so small that
    \begin{align*}
        d(r(V_i), G^{(0)}\setminus r(V_i)) > 3 \varepsilon_{n+1} , \quad \text{and} \quad 
        d(s(V_i) , G^{(0)} \setminus s(V_i)) > 2 \varepsilon_{n+1},
    \end{align*}
    along with
    \begin{align}\label{eq:theta-containments}
        \theta_i^{-1}( {(\theta_i (A_n \cap s(V_i)))^{[3\varepsilon_{n+1}]}}) \subseteq {A_n^{[\varepsilon_n]}} \quad \text{and} \quad \theta_i ({F_{n,i,k_0}^{[2\varepsilon_{n+1}]} }) \subseteq {(G^{(0)} \setminus B_{n,k_0})^{[\varepsilon_n]}}.
    \end{align}
    Now define
    \begin{align*}
        U_n := s(V_i) \cap \theta_i^{-1}(B_{n,k_0} \cap (\theta_i( A_n \cap s(V_i)))^{\varepsilon_{n+1}} ),
    \end{align*}
    which we see is open. We may now specify the sets for the next step of the algorithm. Set
    \begin{align}
        A_{n+1} := A_n \setminus U_n , \quad \text{and} \quad B_{n+1,k} := \begin{cases}
            B_{n,k} & k\neq k_0 \\
            B_{n,k_0}\setminus \overline{\theta_i(U_n)} & k=k_0. 
        \end{cases}
    \end{align}
    We moreover associate to $U_n$ the open bisection
    \begin{align}\label{eq:restricted-Uns}
        \widetilde{U}_n = (s\vert_{V_i})^{-1}(U_n) \subseteq V_i.
    \end{align}
    It now remains to show that \eqref{eq:algo-LHS-vs-RHS} is still satisfied, that the algorithm stops after $Nm$ steps, and that we can construct the families $\mathcal{U}_1 , \ldots , \mathcal{U}_m$ from the $Nm$ open bisections defined by \eqref{eq:restricted-Uns}. We do so in order. 

    Fix $x \in G^{(0)}$, and suppose that $\gamma \in R_{n, k_0}(x) \setminus R_{n+1, k_0}(x)$. Then 
    \begin{align}\label{eq:rgamma-away}
        d(r(\gamma) , G^{(0)}\setminus B_{n,k_0}) > \varepsilon_n
    \end{align}
    and since $G^{(0)} \setminus B_{n+1,k_0} = (G^{(0)} \setminus B_{n,k_0}) \cup \overline{\theta_i(U_n)}$, 
    the fact that $\gamma \not \in R_{n+1, k_0}(x)$ implies $d(r(\gamma), (G^{(0)} \setminus B_{n,k_0}) \cup \overline{\theta_i(U_n)}) \leq \varepsilon_{n+1}.$
   Together with \eqref{eq:rgamma-away} and $\varepsilon_{n+1} < \varepsilon_n$, we obtain
   $d(r(\gamma), \overline{\theta_i(U_n)}) \leq \varepsilon_{n+1}$. 

   Now, by \eqref{eq:restricted-Uns} we have
   $\overline{\theta_i(U_n)} \subseteq {(\theta_i(A_n \cap s(V_i)))^{[\varepsilon_{n+1}]}}$, and so 
   \begin{align*}
       r(\gamma) \in {(\theta_i(A_n \cap s(V_i)))^{[2\varepsilon_{n+1}]}}.
   \end{align*}
   Using 
   \begin{align*}
       d(r(V_i), G^{(0)} \setminus r(V_i))>3\varepsilon_{n+1},
   \end{align*}
   it follows that 
   $r(\gamma) \in r(V_i)$. 
   It follows that there is a unique $v(\gamma) \in V_i$ with $r(v(\gamma)) = r(\gamma)$, and we define
   \begin{align*}
       z(\gamma) := s(v(\gamma)) = \theta_i^{-1}(r(\gamma)). 
   \end{align*}
   We noted above that $r(\gamma) \in {(\theta_i(A_n \cap s(V_i)))^{[2\varepsilon_{n+1}]}}$, which combined with \eqref{eq:theta-containments} implies $z(\gamma) \in {A_n^{[\varepsilon_n]}}$. We claim that $z(\gamma) \not \in {A_{n+1}^{[\varepsilon_{n+1}]}}$.

    Suppose this is not the case. Since $A_{n+1}$ is compact, there must be $a \in A_{n+1} \subseteq A_n$ with $d(a, z(\gamma)) \leq \varepsilon_{n+1}$, which implies $a \in s(V_i)$. If $\theta_i(a) \in B_{n,k_0}$,  then
    \begin{align*}
    \theta_i(a) \in B_{n, k_0} \cap (\theta_i(A_n \cap s(V_i)))^{\varepsilon_{n+1}},
    \end{align*}
    from which we deduce $a \in U_n$. But this contradicts $a \in A_{n+1} = A_n \setminus U_n$. We therefore have $\theta_i(a) \in G^{(0)} \setminus B_{n,k_0}$ and thus $a \in F_{n,i,k_0}$. It then follows from $d(a, z(\gamma)) \leq \varepsilon_{n+1}$ and \eqref{eq:theta-containments} that $r(\gamma) = \theta_i(z(\gamma)) \in {(G^{(0)} \setminus B_{n,k_0})^{[\varepsilon_n]}}$, which contradicts $d(r(\gamma), G^{(0)}\setminus B_{n,k_0}) > \varepsilon_n $. Thus 
    $z(\gamma)\notin A_{n+1}^{[\varepsilon_{n+1}]}$. 
    
    We now define 
    \begin{align*}
        \iota_{n,x} \colon R_{n,k_0}(x) \setminus R_{n+1, k_0}(x) &\to L_n(x) \setminus L_{n+1}(x) \\
        \gamma &\mapsto v(\gamma)^{-1}\gamma.
    \end{align*}
    Note that $v(\gamma)^{-1}\gamma \in D^{-1}Dx$ and $r( v(\gamma)^{-1}\gamma) = s(v(\gamma)) = z(\gamma)$. Since we noted $z(\gamma) \in {A_n^{[\varepsilon_n]}}\setminus {A_{n+1}^{[\varepsilon_{n+1}]}}$, it follows that $\iota_{n,x}$ is a well-defined map. Furthermore, if $\iota_{n,x}(\gamma_1) = \iota_{n,x}(\gamma_2)= \gamma$,  then $s(v(\gamma_1)) = r(\gamma) = s(v(\gamma_2))$, from which we deduce that $v(\gamma_1) = v(\gamma_2)$ since $V_i$ is a bisection, and thus $\gamma_1 = \gamma_2$. It follows that $\iota_{n,x}$ is an injection, and thus
    \begin{align*}
        \vert R_{n,k_0}(x) \setminus R_{n+1, k_0}(x) \vert \leq \vert L_n(x) \setminus L_{n+1}(x) \vert. 
    \end{align*}
    Since $A_{n+1} \subseteq A_n$ and $\varepsilon_{n+1} < \varepsilon_n$, we have $L_{n+1}(x) \subseteq L_n(x)$. For $k \neq k_0$, we have $B_{n+1,k} = B_{n,k}$, so $\varepsilon_{n+1} < \varepsilon_n$ yields $R_{n,k}(x) \subseteq R_{n+1, k}(x)$. Moreover
    \begin{align*}
        \vert R_{n+1,k_0}(x)\vert \geq \vert R_{n,k_0}(x)\vert -\vert R_{n,k_0}(x)\setminus R_{n+1,k_0}(x)\vert.
    \end{align*}
    Then it follows that 
    \begin{align*}
        \vert L_{n+1}(x)\vert 
 &=\vert L_n(x)\vert-\vert L_n(x)\setminus L_{n+1}(x)\vert\\
 &<\sum_{k=1}^m|R_{n,k}(x)\vert  -\vert L_n(x)\setminus L_{n+1}(x)\vert\\
 &\leq\sum_{k=1}^m|R_{n+1,k}(x) \vert
    \end{align*}
    and we deduce that \eqref{eq:algo-LHS-vs-RHS} is still satisfied. 

    We claim that the algorithm stops after $Nm$ steps, that is, $A \subseteq \cup_{n=1}^{Nm} U_n$. If this is not the case, we may choose $x \in A_{Nm+1}$. Since $G^{(0)} \subseteq D$, we have $x \in  D^{-1}Dx$, and $x = r(x) \in {A_{Nm+1}^{[\varepsilon_{Nm+1}]}}$. Thus $ \vert L_{Nm+1}(x) \vert \geq 1$. By \eqref{eq:algo-LHS-vs-RHS} there are $k \in \{ 1, \ldots , m \}$ and $\gamma \in Dx$ such that $r(\gamma) \in B_{Nm+1,k}^{-\varepsilon_{Nm+1}}$. Choose  $V_i$ with $\gamma \in V_i$, and let $n$ be the unique element of $\{1, \ldots, Nm\}$ for which $\sigma(n) = (i,k)$. The sets $A_j$ and $B_{j,k}$ decrease with increasing $j$, and so $x \in A_n$ and $r(\gamma) \in B_{n,k}$. Since $\gamma \in V_i$ and $s(\gamma) = x$, we have $x \in s(V_i)$, 
    and $\theta_i(x) = r(\gamma)$. Then $\theta_i(x) \in B_{n,k} \cap (\theta_i(A_n \cap s(V_i)))^{\varepsilon_{n+1}}$, which implies $x \in U_n$. But this is a contradiction to $x \in A_{Nm+1} \subseteq A_{n+1} = A_n \setminus U_n$, and so we deduce that indeed $A \subseteq \cup_{n=1}^{Nm}U_n$ and the algorithm stops.

    It now only remains to specify the families $\mathcal{U}_1 , \ldots , \mathcal{U}_m$. To do so we consider the open bisections $(\widetilde{U}_n)_{n=1}^{Nm}$ from \eqref{eq:restricted-Uns}. We have $s(\widetilde{U}_n) = U_n$, so $A \subseteq \cup_{n=1}^{Nm} s(\widetilde{U}_n)$. Moreover, 
    \begin{align*}
        r(\widetilde{U}_n) = \theta_{i(n)}(U_n) \subseteq B_{n, k(n)} \subseteq B. 
    \end{align*}
    If $n < n'$ and $k(n) = k(n')$, then
    \begin{align*}
        r(\widetilde{U}_{n'})   \subseteq B_{n+1,k(n)}
 \subseteq B_{n,k(n)}\setminus \overline{r(\widetilde{U}_n})
    \end{align*}
    from which $r(\widetilde{U}_n) \cap r(\widetilde{U}_{n'}) = \emptyset$.
    For $j =1 , \ldots , m$ we therefore set
    \begin{align*}
        \mathcal{U}_{j} = \{ \widetilde{U}_{\sigma^{-1}(i,j)} \}_{i=1}^N ,
    \end{align*}
    which implements $A \precsim_{m-1} B$.  
\end{proof}

\begin{theorem}\label{thm:main-theorem-wocg}
    Let $G$ be a locally compact Hausdorff étale groupoid with compact and metrizable unit space. Suppose $G$ admits a comparison admissible uniformly fibrewise coarse length function $\ell$ with admissibility constant $\Lambda$. Suppose moreover that $\ell$ has bounded doubling with doubling constant $C$. Then $G$ has weak $(m-1)$-comparison where
    \begin{align*}
        m = \lceil 4 C^{\lceil \log_2(8\Lambda)\rceil} \rceil.
    \end{align*}
\end{theorem}
\begin{proof}
    Lemma \ref{lemma:density-approx} gives $M(G) \neq \emptyset$. 
    Let $A \subseteq G^{(0)}$ be closed and $B \subseteq G^{(0)}$ be open with $A \subseteq r(GB)$, and suppose 
    \begin{align*}
        \sup_{\mu \in M(G)} \mu(A) < \inf_{\mu \in M(G)} \mu(B).
    \end{align*}
    Applying Lemma \ref{Lemma:an epsilon of room} yields a constant $\delta > 0$ for which
    \begin{align*}
        \sup_{\mu \in M(G)} \mu ( {A^{[\delta]}} ) < \inf_{\mu \in M(G)} \mu(B^{-\delta}).
    \end{align*}
    Then Lemma \ref{lemma:density-approx} yields the existence of $R_0\geq 0$ such that for all $R \geq R_0$ and $x,y \in G^{(0)}$ we have
    \begin{align*}
        \frac{N_{4\Lambda R} ({A^{[\delta]}},x)}{V_{4\Lambda R}(x)} < \frac{N_R(B^{-\delta},y)}{V_R(y)}.
    \end{align*}
    Choose an $R \geq \max \{ R_0, 1 \}$ big enough that comparison admissibility of $\ell$ yields a $D_R$ with $B_\ell(R) \subseteq D_R \subseteq B_\ell(\Lambda R)$. By subadditivity and symmetry of the length function, we have $D_R^{-1} D_R \subseteq B_\ell(2\Lambda R)$. Note that for $\gamma \in B_\ell(2\Lambda R)$ we have $4\Lambda R - \ell(\gamma) \geq 2\Lambda R$. We thus calculate
    \begin{align*}
        \vert \{ \gamma \in D_R^{-1}D_R x \mid r(\gamma) \in {A^{[\delta]}} \} \vert &\leq \vert \{ \gamma \in B_\ell(2 \Lambda R) x \mid r(\gamma) \in {A^{[\delta]}} \}\vert  \\
        &= \frac{2\Lambda R \vert \{ \gamma \in B_\ell(2 \Lambda R)x \mid r(\gamma) \in {A^{[\delta]}} \}\vert}{2\Lambda R} \\
        &\leq \frac{1}{2 \Lambda R} \sum_{\gamma \in B_\ell(2\Lambda R)x}(4\Lambda R - \ell(\gamma)) 1_{{A^{[\delta]}}}(r(\gamma)) \\
        &\leq \frac{N_{4\Lambda R}({A^{[\delta]}}, x)}{2\Lambda R}.
    \end{align*}
    Since $B_\ell(R) \subseteq D_R$ we similarly calculate
    \begin{align*}
        \frac{N_R(B^{-\delta}, x)}{R} &= \frac{1}{R} \sum_{\gamma \in B_\ell(R)x} (R - \ell(\gamma)) 1_{B^{-\delta}}(r(\gamma)) \\
        &\leq \sum_{\gamma \in B_\ell(R)x} 1_{B^{-\delta}}(r(\gamma)) \\
        &= \vert \{ \gamma \in B_\ell(R)x \mid r(\gamma) \in B^{-\delta} \} \vert \\
        &\leq \vert \{\gamma \in D_R x \mid r(\gamma) \in B^{-\delta} \} \vert. 
    \end{align*}
    Using the doubling condition and the trivial estimates
    \begin{align*}
        V_R(x) &= \sum_{\gamma \in B_\ell(R) x} (R-\ell(\gamma)) \geq \frac{R}{2}\vert B_\ell(\frac{R}{2}) x \vert \\
        V_{4\Lambda R}(x) &= \sum_{\gamma \in B_\ell(4\Lambda R)x}(4\Lambda R - \ell(\gamma)) \leq 4\Lambda R \vert B_\ell(4\Lambda R) x\vert  ,
    \end{align*}
    we calculate
    \begin{align*}
        \frac{V_{4\Lambda R}(x)}{V_R(x)} \leq \frac{4\Lambda R \vert B_\ell(4 \Lambda R)x \vert}{\frac{R}{2} \vert B_\ell(\frac{R}{2}) x\vert } \leq 8 \Lambda C^{\lceil \log_2 (8 \Lambda) \rceil}. 
    \end{align*}
    Then
    \begin{align*}
        \vert \{ \gamma \in D_R^{-1} D_R x \mid r(\gamma) \in {A^{[\delta]}} \} \vert &\leq \frac{N_{4\Lambda R}({A^{[\delta]}},x)}{2\Lambda R} < \frac{V_{4\Lambda R}(x)}{2\Lambda R} \cdot \frac{N_R(B^{-\delta},x)}{V_R(x)} \\
        &\leq 4C^{\lceil \log_2 (8 \Lambda)\rceil} \cdot \frac{N_R(B^{-\delta}, x)}{R} \\
        &\leq 4 C^{\lceil \log_2 (8 \Lambda)\rceil}  \vert \{ \gamma \in D_R x \mid r(\gamma) \in B^{-\delta} \} \vert .
    \end{align*}
    Now Lemma \ref{Lemma: m-comparison} gives $A \precsim_{m-1} B$ with $m = \lceil 4 C^{\lceil \log_2(8\Lambda)\rceil} \rceil$. Since $A\subseteq G^{(0)}$ closed and $B \subseteq G^{(0)}$ open with $A \subseteq r(GB)$ were arbitrary, we deduce that $G$ satisfies weak $(m-1)$-comparison with $m = \lceil 4 C^{\lceil \log_2(8\Lambda)\rceil} \rceil$.
\end{proof}
In order to apply Theorem \ref{thm:main-theorem-wocg} in the context of homological invariants for topological full groups, 
we will be particularly interested in the case when $G$ is ample. It then turns out that comparison admissibility of a length function is guaranteed under quite natural assumptions, as the following lemma shows. 
\begin{lemma}\label{lemma:comp-admissible-ample-existence}
    Let $G$ be a locally compact Hausdorff ample groupoid with compact unit space $G^{(0)}$. If $\ell \colon G \to [0,\infty)$ is a continuous proper length function, then $\ell$ is comparison admissible for any $\Lambda > 1$.
\end{lemma}
\begin{proof}
    Fix $\Lambda > 1$ and $R > 0$. For every $\gamma \in B_\ell(R)$, the set $\{\mu \in G \mid \ell(\mu) < \Lambda R \}$ is an open neighbourhood of $\gamma$. As $G$ is ample, it must contain a compact open bisection $V_\gamma$ with $\gamma \in V_\gamma$. Note that $V_\gamma \subseteq B_\ell(\Lambda R)$.  
    Since $\ell$ is proper and $G^{(0)}$ is compact it follows that $B_\ell(R)$ is compact. We therefore find a finite subcover of the $V_\gamma$'s, say $V_1 ,\ldots ,V_N$. Since each $V_i \subseteq B_\ell(\Lambda R)$, we have
    \begin{align*}
        B_\ell(R) \subseteq \bigcup_{i=1}^N V_i \subseteq B_\ell(\Lambda R),
    \end{align*}
    showing that $\ell$ is comparison admissible.
\end{proof}

We will also need the following lemma allowing us to reduce from weak $m$-comparison to genuine comparison under the assumption of minimality and ampleness of the groupoid. It adapts the proof of \cite[Proposition~3.10]{AraBonickeBosaLi23} in the same way \cite[Lemma~2.3]{NaryshkinComparison2022} adapts the proof of \cite[Lemma~6.4]{DownarowiczZhang23}.

\begin{lemma}\label{lemma: weak m-comp implies comparison}
    Let $G$ be a $\sigma$-compact locally compact Hausdorff minimal ample groupoid with compact and metrizable unit space. Suppose $M(G) \neq \emptyset$.  
    If $G$ has weak $m$-comparison for some $m\in \mathbb N \cup \{0\}$, then $G$ has comparison.
\end{lemma}
\begin{proof} Note that any minimal groupoid with a finite  unit space has comparison. We can therefore assume that $\Go$ is infinite. 
Minimality then implies that $\Go$ has no isolated points, because an isolated point would belong to every dense orbit, so the entire unit space would be a single discrete orbit, hence finite by compactness. 
Every orbit is infinite, and invariance assigns equal mass to its points, so every invariant probability measure is nonatomic. For a decreasing clopen neighbourhood basis $(W_j)$ at any point $x$, the continuous functions $\mu\mapsto\mu(W_j)$ decrease pointwise to zero on the compact space $M(G)$. Dini's theorem gives $\sup_\mu\mu(W_j)\to0$. 
As noted in \cite{AraBonickeBosaLi23}, it suffices to show the result under the assumption that $A$ and $B$ are clopen subsets of $\Go$. 
Let $A,B \subseteq \Go$ be clopen with
 $$\mu(A)<\mu(B) \hbox{ for all }\mu\in M(G).$$
    It follows from \cite[Remark~3.8]{AraBonickeBosaLi23} that there exists an $\varepsilon > 0 $ such that $\mu(B)-\mu(A) > \varepsilon$ for all $\mu\in M(G)$.
    Note also that $B$ must be infinite. 
    Fix a point $x\in \Go$. We claim that we can inductively find elements $\gamma_0,\ldots, \gamma_m\in G$ such that $s(\gamma_i)=x$, $r(\gamma_i)\in B$, and $r(\gamma_i)\neq r(\gamma_j)$ unless $i=j$. Indeed, since $G$ is minimal and $B$ is open we have $ B\cap r(Gx)\neq \emptyset$. Thus, there exists an element $\gamma_0\in G$ such that $s(\gamma_0)=x$ and $r(\gamma_0)\in B$.
    Supposing now that we have already found $\gamma_0,\ldots, \gamma_k$, use the fact that $B$ is infinite and Hausdorff to find a proper clopen subset $O\subset B$ such that $r(\gamma_0),\ldots, r(\gamma_k)\in B\setminus O$. 
    Since $O$ is open and $G$ is minimal, $O\cap r(Gx)\neq \emptyset$ and hence there exists an element $\gamma_{k+1}\in G$ such that $s(\gamma_{k+1})=x$ and $r(\gamma_{k+1})\in O$. Then the elements $\gamma_0,\ldots, \gamma_k,\gamma_{k+1}\in G$ satisfy the desired requirements.
    
    For each $0\leq i\leq m$ find a compact open bisection $S_i\subseteq G$ such that $\gamma_i\in S_i$. By shrinking the $S_i$ if necessary we may assume without loss of generality that $s(S_i)=s(S_j)$ for all $i,j\in \{0,\ldots, m\}$, that for $U:=s(S_0)$ we have $\sup_{\mu\in M(G)}\mu(U)<\frac{\varepsilon}{m+1}$, and that the sets $r(S_i)$ are pairwise disjoint subsets of $B$.
    Since $G$ is minimal, every invariant probability measure has full support as finitely many bisection translates of a nonempty open set cover the compact unit space, so that set cannot have measure zero. Since $U$ is clopen, $\mu\mapsto\mu(U)$ is continuous on compact $M(G)$, and therefore
    $$\inf_{\mu\in M(G)}\mu(U)>0.$$
    Define 
    $$B_1 \coloneqq B\setminus (r(S_0)\sqcup \ldots \sqcup r(S_m)).$$
    Note that $B_1$ is a clopen subset of $\Go$ such that for all $\mu\in M(G)$ we have
    \begin{align*}
        \mu(B_1)&=\mu(B)-\sum_{i=0}^m\mu(r(S_i))=\mu(B)-\sum_{i=0}^m\mu(s(S_i))\\
        &=\mu(B)-(m+1)\mu(U)\geq \mu(B)-\varepsilon >\mu(A).
    \end{align*}
 
    Set $A_1\coloneqq A$. It follows from \cite[Remark~3.8]{AraBonickeBosaLi23} again that there exists a $\delta > 0 $ such that 
    \begin{equation}\label{eq:comparison+delta}
        \mu(B_1)-\mu(A_1) > \delta \hbox{ for all }\mu\in M(G).
    \end{equation} Using $\sigma$-compactness, we can find a countable cover $(V_n)_{n\in \mathbb N}$ of $G$ by compact open bisections. Let $\theta_n \colon s(V_n)\to r(V_n)$ denote the homeomorphism associated with the bisection $V_n$. Inductively, define sequences $(U_n)_n$, $(A_n)_n$ and $(B_n)_n$ as follows:
    \begin{align*}
        U_n & \coloneqq \theta_n^{-1}(B_n\cap \theta_n(A_n \cap s(V_n))), \\
        B_{n+1} & \coloneqq B_n\setminus \theta_n(U_n), \hbox{ and }\\
        A_{n+1} & \coloneqq A_n\setminus U_n.
    \end{align*}
Note that each of the sets above is compact open. Moreover, note that it may happen that $U_n$ is empty for some $n$, in which case $B_{n+1}=B_n$ and $A_{n+1}=A_n$, so the algorithm can still continue from there.
We claim that for all $n\in \mathbb N$ we have that $\mu(A_n)+\delta<\mu(B_n)$ for all $\mu\in M(G)$. We already know from (\ref{eq:comparison+delta}) that this is the case for $n=1$. Proceeding by induction, we have
$$\mu(B_{n+1})=\mu(B_n)-\mu(\theta_n(U_n))=\mu(B_n)-\mu(U_n)=\mu(B_n)-\mu(A_n)+\mu(A_{n+1})>\delta + \mu(A_{n+1}).$$
Let $A_\infty = \bigcap_{n\in \mathbb N} A_n$ and $B_\infty = \bigcap_{n\in \mathbb N} B_n$. As the sequences $(A_n)_n$ and $(B_n)_n$ are decreasing, we conclude that
$$\mu(B_\infty)=\lim_{n\to \infty}\mu(B_n) \geq \lim_{n\to \infty} \mu(A_n) + \delta =\mu(A_\infty)+\delta.$$
We claim that
$$
\sup_{\mu\in M(G)} \mu(A_\infty)=0.
$$
To see this, note that we have $r(G A_\infty) \cap B_\infty = \emptyset$. 

Let $\mu\in M(G)$ be an ergodic measure. Since $\mu(B_\infty)>0$ we have $1\geq \mu(r(GA_\infty))+\mu(B_\infty)>\mu(r(GA_\infty))$. Hence ergodicity of $\mu$ implies that $\mu(r(GA_\infty))=0$, and since $A_\infty\subseteq r(GA_\infty)$ we conclude that $\mu(A_\infty)=0$. 
The claim for arbitrary measures in $M(G)$ follows by ergodic decomposition. 

For each $n\in \mathbb N$ consider the function
$$f_n \colon M(G)\to [0,1], \ f_n(\mu)=\mu(A_n).$$
Since $A_n$ is compact open, the functions $f_n$ are continuous for all $n\in \mathbb N$, and as $\mu(A_\infty)=0$ for all $\mu\in M(G)$, the sequence $(f_n)_n$ converges pointwise to the constant zero function. As $M(G)$ is compact, Dini's Theorem applies and we conclude that this convergence is in fact uniform. In particular, there exists a $k\in \mathbb N$ such that $f_k(\mu)< \inf_{\mu\in M(G)} \mu(U)$ for all $\mu\in M(G)$, and hence
$$\sup_{\mu\in M(G)} \mu(A_k) < \inf_{\mu\in M(G)} \mu(U).$$
By weak $m$-comparison, $A_k\precsim_m U$ which implies
\begin{equation}\label{eq: A_k prec B-B_1}
    A_k\precsim \bigsqcup_{i=0}^m r(S_i)=B\setminus B_1.
\end{equation}

By construction, $A\setminus A_k=\bigsqcup_{i=1}^{k-1} U_i$ and the sets $\theta_i(U_i)\subseteq B_1$ are pairwise disjoint for all $1\leq i\leq k-1$. In other words we also have 
$A\setminus A_k\precsim B_1$. Combining this with (\ref{eq: A_k prec B-B_1}) yields $A\precsim B$ as desired.
\end{proof} 

In \cite{AraBonickeBosaLi23} it was shown that for minimal ample groupoids, the comparison property can be characterised algebraically in terms of the type semigroup of $G$ introduced in \cite{BL20}. In this language we obtain the following immediate consequence of Theorem \ref{thm:main-theorem-wocg} and Lemma \ref{lemma: weak m-comp implies comparison}.
\begin{corollary}\label{cor:almost-unperforation}
    Let $G$ be a minimal  locally compact Hausdorff ample groupoid with compact and metrizable unit space. If $G$ admits a comparison admissible uniformly fibrewise coarse length with uniformly bounded doubling, then its type semigroup $S(G)$ is almost unperforated.
\end{corollary}
As an application of Theorem \ref{thm:main-theorem-wocg} we can conclude that the 
comparison and minimality hypotheses in the results relating homological invariants of topological full groups and groupoid homology
\cite{LiInfLoopSpaces25} are satisfied for locally compact Hausdorff minimal ample groupoids with 
compact metrizable unit space without isolated points, admitting a comparison admissible uniformly fibrewise coarse length function with uniformly bounded doubling. 
As a detailed discussion of these methods is beyond the scope of the article, we only explicitly state a sample result illustrating one of the many strong consequences one can obtain by combining  Theorem \ref{thm:main-theorem-wocg} and Lemma \ref{lemma: weak m-comp implies comparison} with the results from \cite{LiInfLoopSpaces25}. 
\begin{corollary}\label{cor:AH-conjecture}
    Let $G$ be a locally compact Hausdorff
    minimal ample groupoid whose unit space is compact and metrizable and has no isolated points. 
    If $G$ admits a uniformly fibrewise coarse length function which is comparison admissible and has bounded doubling, then $G$ satisfies Matui's AH-conjecture.  
\end{corollary}

\begin{proof}
    By Theorem \ref{thm:main-theorem-wocg}, $G$ has weak $m$-comparison for some sufficiently large $m$. 
    By Remark \ref{remark:comparison-admissible-implies-sigma-compact}, $G$ is $\sigma$-compact, and Lemma \ref{lemma:density-approx} gives $M(G) \neq \emptyset$. Since $G$ is minimal and ample, Lemma \ref{lemma: weak m-comp implies comparison} implies that $G$ has comparison. Then \cite[Corollary E]{LiInfLoopSpaces25} yields that $G$ satisfies Matui's AH-conjecture. 
    
\end{proof}

\subsection{Partial actions}
Let $\Gamma$ be a countable group and let $X$ be a compact metrizable space. Suppose $\theta = (\{X_g\}_{g \in \Gamma}, \{\theta_g \}_{g\in \Gamma})$ is a partial action of $\Gamma$ on $X$, that is
    \begin{enumerate}
        \item For all $g \in \Gamma$ we have $X_g \subseteq X$ is open and $\theta_g \colon X_{g^{-1}} \to X_g$ is a homeomorphism;
        \item $D^{-1} := \{ (g,x) \in \Gamma \times X \mid g \in \Gamma, x \in X_{g^{-1}} \} $ is open in $\Gamma \times X$, and $D^{-1}\ni (g,x) \mapsto \theta_g(x) \in X$ is continuous;
        \item $X_e = X$ and $\theta_e = \operatorname{id}_X$,  where $e$ is the unit of $\Gamma$, and for all $g_1,g_2 \in \Gamma$ the map $\theta_{g_1g_2}$ is an extension of $\theta_{g_1}\theta_{g_2}$.
    \end{enumerate}
    To such a partial action we may associate a groupoid $G:= X\rtimes_\theta \Gamma$, known as the \emph{partial action groupoid associated with $\theta$}. As a set we have
    \begin{align*}
        X\rtimes_\theta \Gamma  = \{ (x,g,y) \in X \times \Gamma \times X \mid y \in X_{g^{-1}} \text{ and } \theta_g(y ) = x \}.
    \end{align*}
    Multiplication and inversion are given by
    \begin{align*}
        (x,g_1,y) (y,g_2, z) = (x, g_1 g_2 , z) \quad \text{and} \quad (x,g,y)^{-1} = (y, g^{-1} ,x ).
    \end{align*}
    We then have a natural identification of $\Go $ with $X$ through $(x,e,x) \mapsto x$. By equipping $G$ with the relative topology from $X \times \Gamma \times X$, we obtain a second-countable locally compact Hausdorff \'etale groupoid \cite{AbadiePartialActions04}.

    Suppose $\Gamma$ comes equipped with a proper length function $\ell_\Gamma$. 
    It is straightforward to verify that $\ell \colon G \to [0,\infty)$ given by
    \begin{align}\label{eq:partial-action-gpd-length}
        \ell(x,g,y) = \ell_\Gamma (g) 
    \end{align}
    is a uniformly fibrewise coarse length function on $G$.
    Moreover, given any $y \in X = (X\rtimes_\theta \Gamma)^{(0)}$, we have
    \begin{align}\label{eq:def-S_y}
        (X\rtimes_\theta \Gamma)_y=\{(\theta_g(y),g,y)\mid y\in X_{g^{-1}} \}\cong \{g\in \Gamma\mid y\in X_{g^{-1}}\}=:S_y.
    \end{align}
    To establish uniformly bounded doubling it is thus often useful to assume that the sets $S_y$ are sufficiently regular in the following sense:

    \begin{definition}
        Suppose $\theta = (\{X_g\}_{g \in \Gamma}, \{\theta_g \}_{g\in \Gamma})$ is a partial action of a discrete group $\Gamma$ on $X$. Let $\ell_\Gamma$ be a length function on $\Gamma$. Then $\theta$ is called \emph{semi-saturated} if whenever $\ell_\Gamma(gh)=\ell_\Gamma(g)+\ell_\Gamma(h)$, then $\theta_{gh}=\theta_g \theta_h$.
    \end{definition}

    This concept is taken from \cite[Theorem~4.9]{Exel17} and is a fairly natural assumption that frequently appears in the literature on partial actions. For a partial $\mathbb Z$-action on a space $X$ for example, semi-saturation with respect to the standard word length is equivalent to saying that the partial action is generated by a single partial homeomorphism $\theta:U\to V$ between open subsets $U$ and $V$ of $X$. 

    \begin{proposition}\label{Prop:semi-saturated Z^d actions have uniform bounded doubling}
        Let $\theta$ be a partial action of $\mathbb Z^d$ on a locally compact Hausdorff space $X$ which is semi-saturated with respect to the standard word length on $\mathbb Z^d$. Then the associated \'etale groupoid $X\rtimes_\theta \mathbb Z^d$, equipped with the induced standard word length, has uniformly bounded doubling.
    \end{proposition}
    \begin{proof}
        For a point $x \in  X$, we define $S_x$ as in \eqref{eq:def-S_y}. 
        Note that since $\mathbb Z^d$ is abelian and the action is semi-saturated, we have that whenever $\ell(gh)=\ell(g)+\ell(h)$ and $gh\in S_x$, then $g,h\in S_x$. 
        Now let $F=\{s_1,\ldots, s_{2d}\}$ be the standard symmetric generating set for $\mathbb Z^d$ (in some arbitrary order). %
        Every $g\in\mathbb Z^d$ has a unique expression
        $g=\sum_{i=1}^{2d}n_i(g)s_i$, where $n_i(g)\in\mathbb N_0$ and
        $\sum_{i=1}^{2d}n_i(g)=\ell(g)$.
        Now let
        \begin{align*}
            q(g)=\lfloor \frac{n_1(g)}{2}\rfloor s_1 + \cdots + \lfloor \frac{n_{2d}(g)}{2}\rfloor s_{2d}
            , \text{ and }p(g)= \lceil \frac{n_1(g)}{2}\rceil s_1+\cdots + \lceil \frac{n_{2d}(g)}{2}\rceil s_{2d}.
        \end{align*}
        Then $g=p(g)+q(g)$, $\ell(q(g))=\sum_{i=1}^{2d} \lfloor \frac{n_i(g)}{2}\rfloor$ and $\ell(p(g))=\sum_{i=1}^{2d} \lceil \frac{n_i(g)}{2}\rceil$ so that $\ell(g)=\ell (p(g)+q(g))=\ell(p(g))+\ell(q(g))$. 
        It follows from semi-saturation that $g\in S_x$ implies $q(g)\in S_x$.
        For any real $r \geq 0$, if $\ell(g) \leq 2r$, then $\ell(q(g)) \leq r$.
        Thus, $q$ restricts to a map
        $$q:S_x\cap B_\Gamma(2r)\to S_x\cap B_\Gamma(r).$$ 
        We claim that the fibres of $q$ are uniformly bounded, more precisely we have $\vert q^{-1}(q(g))\vert \leq 2^{2d}$.
        To see this write $n_i(g)=2 \lfloor \frac{n_i(g)}{2}\rfloor +\varepsilon_i(g)$ for $\varepsilon_i(g)\in \{0,1\}$. Then we have $g=2q(g)+r(g)$ where $r(g)=\sum_{i=1}^{2d} \varepsilon_i(g) s_i$. Thus, each $h\in q^{-1}(q(g))$ is uniquely determined by $r(h)$. The element $r(h)$ in turn is determined by the choice of $\varepsilon_i(h)\in \{0,1\}$ for $1\leq i\leq 2d$ and hence we conclude that $\vert q^{-1}(q(g))\vert \leq 2^{2d}$. 
        It follows that
        $$\vert S_x\cap B_\Gamma(2r)\vert = \sum_{h\in S_x\cap B_\Gamma(r)} \vert\{g\in S_x\cap B_\Gamma(2r)\mid q(g)=h\}\vert \leq 2^{2d} \vert S_x\cap B_\Gamma(r)\vert.$$
    \end{proof}
    
    We also record the following result which might be of independent interest. 
    \begin{proposition}\label{prop:partial-action-gpd-length-comp-adm}
         Let $\theta$ be a partial action of a countable discrete group $\Gamma$ on a compact Hausdorff space $X$, and suppose $\ell_\Gamma$ is a proper length function on $\Gamma$. Let $\ell$ be the length function induced through \eqref{eq:partial-action-gpd-length}. 
         Then $\ell$ is comparison admissible if and only if all
         the domains $\{X_g\}_{g \in \Gamma}$ are closed. In particular, if the action is global, $\ell$ is always comparison admissible on the resulting transformation groupoid. 
    \end{proposition}
    \begin{proof}
    Suppose first that all the domains $\{X_g\}_{g\in \Gamma}$ are closed.
        For each $g \in \Gamma$, the set
        \begin{align*}
            V_g = \{(x,g,y) \mid y \in X_{g^{-1}} \text{ and } \theta_g(y) = x  \}
        \end{align*}
        is an open bisection. It is homeomorphic to the closed subset $X_{g^{-1}} \subseteq X$ through the source map, and $X$ is compact, so $X_{g^{-1}}$ is compact. Thus $V_g$ is also compact. Since $\ell_\Gamma$ is proper, we have $B_{\ell_\Gamma}(r) = \{g \in \Gamma \mid \ell_\Gamma(g) \leq r \} $ is finite for all $r \geq 0$. Thus
        \begin{align*}
            B_\ell(r) = \bigcup_{g \in B_{\ell_\Gamma}(r)} V_g
        \end{align*}
        showing that $B_\ell(r)$ is a finite union of compact open bisections. Choosing $D_r = B_\ell(r)$ shows that $\ell$ is comparison admissible with $\Lambda = 1$. 

        Conversely, suppose $\ell$ is comparison admissible for some $\Lambda \geq 1$, and fix some $g \in \Gamma$ and an $R \geq \ell_\Gamma (g)$ large enough for comparison admissibility. 
        By comparison admissibility, $D_R$ is compact and contains $B_\ell(R)$, hence contains $V_g$.
        Using the continuous map
        \begin{align*}
            a \colon G \to \Gamma , \quad a(x,g,y) = g,
        \end{align*}
        we realize $V_g = D_R \cap a^{-1}(\{g\})$.  As $\{g\}$ is closed, we deduce $V_g$ is compact, and it follows that $X_{g^{-1}}$ is compact and hence closed in $X$. 

        The statement about global actions follows trivially. 
    \end{proof}
    
   Putting together our observations so far, we may obtain the following consequence of Theorem \ref{thm:main-theorem-wocg}. 

    \begin{corollary}
        Let $\theta$ be a minimal partial action of $\mathbb Z^d$ on a compact metrizable zero-dimensional space $X$, for which the domains of the standard generators are closed. Suppose $\theta$ is semi-saturated with respect to the standard word-length function on $\mathbb Z^d$.  Then the type semigroup $S(X \rtimes_\theta \mathbb Z^d)$ is almost unperforated. If $X$ has no isolated points, then $X \rtimes_\theta \mathbb Z^d$ satisfies Matui's AH conjecture. 
    \end{corollary}

    \begin{proof}
        Every group element has a geodesic word in the standard symmetric generators. Semi-saturation expresses its partial map as the corresponding composition of generator maps. A finite composition of partial homeomorphisms between clopen subsets again has clopen domain and range. Thus every $X_g$ is clopen, and Proposition \ref{prop:partial-action-gpd-length-comp-adm} shows that the induced length $\ell$ on $X\rtimes_\theta\mathbb Z^d$ is comparison admissible. 
        Since $(\mathbb Z^d, \ell_{\mathbb Z^d})$ has bounded doubling for the standard word length $\ell_{\mathbb Z^d}$, Proposition \ref{Prop:semi-saturated Z^d actions have uniform bounded doubling} yields that $(X \rtimes_\theta \mathbb Z^d, \ell)$ has bounded doubling. Then Theorem \ref{thm:main-theorem-wocg} yields that $X \rtimes_\theta \mathbb Z^d$ has weak $m$-comparison for some sufficiently large $m$. Since $\theta$ is minimal and $X$ is zero-dimensional, it follows that $X \rtimes_\theta \mathbb Z^d$ is a minimal ample groupoid with bounded doubling from which we deduce that $X \rtimes_\theta \mathbb Z^d$ has comparison by Lemma \ref{lemma: weak m-comp implies comparison}. It follows that the type semigroup $S(X \rtimes_\theta \mathbb Z^d)$ is almost unperforated by Corollary \ref{cor:almost-unperforation}. If $X$ has no isolated points, then it follows from Corollary \ref{cor:AH-conjecture} that $X \rtimes_\theta \mathbb Z^d$ satisfies Matui's AH conjecture. 
    \end{proof}

We also include the following example which shows that we in general cannot drop the condition of semi-saturation in Proposition \ref{Prop:semi-saturated Z^d actions have uniform bounded doubling} and hope to retain bounded doubling for the length induced by the standard word length on the associated partial action groupoid.

\begin{example}
We consider a partial Bernoulli action as in \cite[Definition~5.12]{Exel17}. Concretely, consider the compact space $\Omega=\{0,1\}^{\mathbb Z}$ with the shift
action $\alpha:\mathbb Z\curvearrowright\Omega$ defined by
\begin{align*}
    (\alpha_n(\omega))(k)=\omega(k+n),
    \qquad n,k\in\mathbb Z,\quad \omega\in\Omega.
\end{align*}
Let $    X=\{\omega\in\Omega:\omega(0)=1\} $. 
Then $X$ is a clopen subset of $\Omega$ and is a Cantor set.
Let $\theta$ denote the partial action obtained by restricting
$\alpha$ to $X$:
\begin{align*}
    X_n:=X\cap\alpha_n(X),
    \qquad
    \theta_n:=\alpha_n|_{X_{-n}}:X_{-n}\longrightarrow X_n.
\end{align*}
Every $X_n$ is clopen in $X$. With our shift convention, $X_{-n}=\{\omega\in X:\omega(n)=1\}$,
and hence
\begin{align*}
    S_\omega
    :=\{n\in\mathbb Z:\omega\in X_{-n}\}
    =\{n\in\mathbb Z:\omega(n)=1\}.
\end{align*}
In particular, for every $A\subseteq\mathbb Z$ containing $0$,
the indicator function $\mathbf{1}_A$ belongs to $X$ and satisfies
$S_{\mathbf{1}_A}=A$.

Equip $G:=X\rtimes_\theta\mathbb Z$ with the length induced by
the standard word length on $\mathbb Z$, namely
\begin{align*}
    \ell\bigl(\theta_n(\omega),n,\omega\bigr)=|n|,
    \qquad \omega\in X_{-n}.
\end{align*}
The identification of source fibres with the sets $S_\omega$ gives
\begin{align*}
    \bigl|B_\ell(r)\mathbf{1}_A\bigr|
    =\bigl|A\cap[-r,r]\bigr|
    \qquad (r\geq0).
\end{align*}

For each integer $n\geq1$, consider
$A_n
    :=\{0\}
      \cup\{-2n,-2n+1,\ldots,-n-1\}
      \cup\{n+1,n+2,\ldots,2n\}$.
Then
\begin{align*}
    A_n\cap[-n,n]=\{0\},
    \qquad
    |A_n\cap[-2n,2n]|=2n+1.
\end{align*}
Consequently,
\begin{align*}
    \frac{\bigl|B_\ell(2n)\mathbf{1}_{A_n}\bigr|}
         {\bigl|B_\ell(n)\mathbf{1}_{A_n}\bigr|}
    =2n+1\longrightarrow\infty.
\end{align*}
Thus there is no constant $C$ such that $|B_\ell(2r)\omega|
    \leq C\,|B_\ell(r)\omega|$ for all $r \geq 0$ and $\omega \in X$. 
Therefore $(G,\ell)$ does not have uniformly bounded doubling,
although $\mathbb Z$ has uniformly bounded doubling for its
standard word length.

The partial action $\theta$ is not semi-saturated with respect
to that word length. Indeed, for
$\omega=\mathbf{1}_{\{0,2\}}$, the map $\theta_2$ is defined at
$\omega$, whereas $\theta_1$ is not. Hence
$\theta_2\neq\theta_1\theta_1$, despite $|2|=|1|+|1|$.
\end{example}

\subsection{Point set groupoids}
To a point set $\Lambda$ in a locally compact group $H$ we may associate a groupoid known as the point set groupoid $G(\Lambda)$, see \eqref{eq:def-GLambda} below. A length function $L$ on $H$ naturally induces a length function $\ell$ on $G(\Lambda)$, \eqref{eq:def-length-ptset-gpd}. In this section we provide conditions on $\Lambda$, $H$ and $L$ which guarantee that $(G(\Lambda), \ell)$ has uniformly bounded doubling. We then further point out properties of $\Lambda$ which allow us to deduce that $G(\Lambda)$ has comparison, has an almost unperforated type semigroup, and satisfies Matui's AH conjecture.

Let $H$ be a second-countable locally compact group equipped with a continuous and proper length function $L \colon H \to [0,\infty)$. In particular, the balls $B_L(r) = \{h \in H \mid L(h) \leq r \}$ are compact for all $r \geq 0$. 
Let $m_H$ be a left Haar measure for $H$, for which we assume the $L$-balls have bounded volume doubling, that is, there is $C_H \geq 1$ such that 
\begin{align}\label{eq:bounded-Haar-volume-doubling}
    m_H(B_L(2r)) \leq C_H m_H (B_L(r)) \quad \text{for all } r \geq 0.
\end{align}
We say a subset $\Lambda \subseteq H$ is \emph{separated} if there exists an open symmetric unit neighborhood $U \subseteq H$ such that $\sup_{x \in H} \vert \Lambda \cap xU \vert \leq 1$, \emph{relatively dense} if there exists a compact set $K$ such that $\vert \Lambda \cap xK \vert \geq 1$ for all $x \in H$, and \emph{Delone} if it is both separated and relatively dense. 

Now fix a separated point set $\Lambda \subseteq H$. A separated subset of $H$ is closed and meets every compact subset in finitely many points. 
Denote by $\mathcal{C}(H)$ the set of closed subsets of $H$. This is a compact topological space when equipped with the Chabauty--Fell topology. Since $H$ is second-countable, so is $\mathcal{C}(H)$ \cite[Section 2.1]{EnstadRaumDynSampling2025}. There is a continuous left action $H \acts \mathcal{C}(H)$ through left multiplication, i.e. we set
\begin{align*}
    xC = \{xy \mid y \in C \}
\end{align*}
for $x \in H$ and $C \in \mathcal{C}(H)$. The closure of the orbit of $\Lambda$ under this action is called the \emph{hull} of $\Lambda$, which we denote by $\Omega(\Lambda)$. The punctured hull is $\Omega^{\times}(\Lambda) := \Omega(\Lambda) \setminus \{\emptyset\}$. 
If $\Lambda$ is relatively dense, then $\emptyset \not \in \Omega(\Lambda)$, so $\Omega(\Lambda) = \Omega^\times (\Lambda)$ in this case, see \cite[Section 2.1]{EnstadRaumDynSampling2025}. 
We further define the transversal of $\Lambda$ to be 
\begin{align*}
    \Omega_0 (\Lambda) := \{ P \in \Omega(\Lambda) \mid e \in P \} \subseteq \Omega^{\times}(\Lambda),
\end{align*}
where $e$ is the identity of $H$. Note that $\Omega_0 (\Lambda)$ is compact as it is closed in $\Omega(\Lambda)$. 

Now consider the transformation groupoid $H \ltimes \Omega^\times(\Lambda)$. The point set groupoid of $\Lambda$, which we will denote $G(\Lambda)$, is the restriction of this transformation groupoid to the transversal $\Omega_0 (\Lambda)$ of $\Lambda$, that is,
\begin{align}\label{eq:def-GLambda}
    G(\Lambda) &= \{ (x,P) \in H \times \Omega^\times(\Lambda) \mid P, xP \in \Omega_0 (\Lambda) \} \notag \\
    &= \{ (x,P) \in H \times \Omega_0 (\Lambda) \mid x^{-1} \in P \}. 
\end{align}
The operations are just inherited from the transformation groupoid, that is
\begin{align*}
     s(x,P)=P,\qquad r(x,P)=xP,
\end{align*}
\begin{align*}
    (x,yP)(y,P)=(xy,P),\qquad
    (x,P)^{-1}=(x^{-1},xP).
\end{align*}
It is shown in \cite[Proposition A.3]{EnstadRaumDynSampling2025} that $G(\Lambda)$ is \'etale when $\Lambda$ is separated. We will in the sequel identify $G(\Lambda)^0$ with $\Omega_0 (\Lambda)$ in the natural way.

We equip $G(\Lambda)$ with a length function $\ell$ given by
\begin{align}\label{eq:def-length-ptset-gpd}
    \ell(x,P ) = L(x).
\end{align}
It is straightforward to verify that $\ell$ is a length function. 
Continuity of $\ell$ follows from it being the composition of the continuous length function $L$ and the projection onto $H$. The incidence relation $x^{-1} \in P$ is closed in $H \times \Omega_0 (\Lambda)$. Specifically, if $(x_i,P_i)\to(x,P)$ and
$x_i^{-1}\in P_i$, then $x^{-1}\in P$ by Chabauty--Fell convergence. 
Thus
\begin{align*}
    B_\ell(R) = G(\Lambda) \cap (B_L(R) \times \Omega_0(\Lambda))
\end{align*}
is compact. 

For the lower counting estimate below, we additionally need
relative density, so we henceforth assume that $\Lambda$ is a Delone set.

\begin{lemma}\label{lemma:separating-and-dense-subsets-Delone}
    Suppose $\Lambda \subseteq H$ is a Delone set. 
    \begin{enumerate}
        \item There exists a relatively compact unit neighbourhood $V \subseteq H$ such that the sets $\{pV \mid p\in P\}$ are pairwise disjoint for every $P \in \Omega_0(\Lambda)$.
        \item There exists a compact symmetric unit neighbourhood $K \subseteq H$ such that $PK = H$ for all $P \in \Omega_0(\Lambda)$. 
    \end{enumerate} 
\end{lemma}

\begin{proof}
    By separation of $\Lambda$ in $H$, there is an open unit neighbourhood $U \subseteq H$ such that $\vert \Lambda \cap xU \vert \leq 1$ for all $x \in H$. Then every translate of $\Lambda$ has the same property with respect to $U$, 
    because for $y \in H$ we have
    \begin{align*}
        \vert y\Lambda \cap xU \vert = \vert \Lambda \cap y^{-1}x U \vert \leq 1.
    \end{align*}
    The subset of all $P \in \mathcal{C}(H)$ which are $U$-separated is closed in the Chabauty-Fell topology by \cite[Proposition 2.2]{EnstadRaumDynSampling2025}. Since $\Omega(\Lambda)$ is obtained as the closure of the translates of $\Lambda$ in the Chabauty-Fell topology, it in particular follows that $\vert P \cap xU \vert \leq 1 $ for all $P \in \Omega_0 (\Lambda)$ and all $x \in H$. 

    Now pick an open, relatively compact identity neighbourhood $V \subseteq H$ such that $VV^{-1} \subseteq U$. Let $p_1 , p_2 \in P$, and suppose that $p_1 V \cap p_2 V \neq \emptyset$. Then there are $v_1, v_2 \in V$ such that $p_1v_1 = p_2 v_2$, and thus
    \begin{align*}
        p_2 = p_1v_1 v_2^{-1} \in p_1VV^{-1} \subseteq p_1 U.
    \end{align*}
    As $U$ is a unit neighbourhood, $p_1 \in p_1U$ as well, and since $P$ is $U$-separated it follows that $p_1=p_2$. Thus (1) holds.

    We proceed to show (2). By relative density there is a compact set $C_1$ such that $\vert\Lambda \cap xC_1 \vert \geq 1$. By local compactness we may also find a compact neighbourhood $ C_2$ of the identity, which contains the unit. Setting $C := C_1 \cup C_1^{-1} \cup C_2 \cup C_2^{-1}$ we see that $C$ is a symmetric unit neighbourhood. We also still have that $\vert \Lambda \cap xC \vert \geq 1$ for all $x \in H$. 
    Note that every translate $y\Lambda$ of $\Lambda$ has the property
    \begin{align*}
        \vert y\Lambda \cap xC \vert = \vert \Lambda \cap y^{-1}xC \vert \geq 1.
    \end{align*}
    By \cite[Proposition 2.2]{EnstadRaumDynSampling2025}, the subset of $C$-dense sets in $\mathcal{C}(H)$ is closed in the Chabauty-Fell topology. Since $\Omega(\Lambda)$ is obtained as the closure of the translates of $\Lambda$, it in particular follows from the above that we have
    \begin{align*}
        \vert P \cap x C \vert \geq 1
    \end{align*}
    for all $P \in \Omega_0 (\Lambda)$ and for all $x \in H$. 
    Note that the condition $P \cap xC \neq \emptyset$ for every $x \in H$ is equivalent to $PC = H$ for all $P \in \Omega_0 (\Lambda)$ because for all $x \in H$ we have
    \begin{align*}
        P \cap xC \neq \emptyset \iff x\in P C^{-1} = PC
    \end{align*}
    for $C = C^{-1}$. Taking $K = C$, item (2) follows. 
 \end{proof}

\begin{proposition}\label{prop:fiberwise-doubling-Delone-pointset-gpd}
    Suppose $\Lambda \subseteq H$ is a Delone set. If \eqref{eq:bounded-Haar-volume-doubling} holds, then there exists $C \geq 0$ such that
    \begin{align*}
        \vert B_\ell (2r) P \vert \leq C \vert B_\ell(r) P \vert 
    \end{align*}
    for all $r \geq 0$ and all $P \in G(\Lambda)^{(0)} = \Omega_0 (\Lambda)$, where $\ell$ is given by \eqref{eq:def-length-ptset-gpd}. 
\end{proposition}

\begin{proof}
     For any $P \in \Omega_0 (\Lambda)$ we  have
    \begin{align*}
        B_\ell(r)P = \{ (x,P) \in G(\Lambda) \mid x^{-1} \in P \text{ and } \ell(x,P) \leq r \} 
    \end{align*}
    which we identify with 
    \begin{align*}
        \{ x \in H \mid x^{-1} \in P \text{ and } L(x) \leq r \} = P^{-1} \cap B_L(r).
    \end{align*}
    Inversion and the symmetry of $L$ give a bijection from this set to $P \cap B_L(r)$, i.e. $\vert B_\ell (r) P \vert = \vert P \cap B_L (r) \vert$ for all $r \geq 0$. 
    So to estimate the growth, and in particular prove doubling, it suffices to consider the set $P \cap B_L(r)$.

    Since $\Lambda$ is Delone, we invoke Lemma \ref{lemma:separating-and-dense-subsets-Delone} to find a relatively compact identity neighbourhood $V \subseteq H$ such that $\{pV \mid p\in P\} $ are pairwise disjoint for every $P \in G(\Lambda)^{(0)}$, and a symmetric compact unit neighbourhood $K \subseteq H$ such that $PK = H$ for every $P \in \Omega_0(\Lambda)$. Now set
    \begin{align*}
        r_V := \sup_{v \in V} L(v) \quad \text{and} \quad r_K := \sup_{k \in K} L(k),
    \end{align*}
   both of which are finite, as $L$ is continuous and thus bounded on compact subsets. We take $V$ open, so $0 < m_H(V) < \infty$ and $0 < m_H (K) < \infty$. 

   To prove bounded doubling, we will find  upper and lower estimates on the quantity $\vert P \cap B_L(r) \vert$, independently of $P \in G(\Lambda)^{(0)}$. First, note that since the sets $pV$, $p \in P$ are all disjoint, we have
   \begin{align*}
       \bigsqcup_{p \in P \cap B_L(r)} pV \subseteq B_L(r+r_V),
   \end{align*}
   from which it follows that
   \begin{align*}
       \vert P \cap B_L(r) \vert \cdot m_H(V) \leq m_H (B_L(r+r_V)),
   \end{align*}
   and thus
   \begin{align*}
       \vert P \cap B_L(r) \vert \leq \frac{m_H (B_L(r+r_V))}{m_H(V)}.
   \end{align*}
   Similarly, we have for all $r \geq r_K$
   \begin{align*}
       B_L(r-r_K) \subseteq (P \cap B_L(r)) K,
   \end{align*}
   since $x = pk \in B_L (r - r_K)$ with $p \in P$ and $k \in K$ implies $L(p) = L(xk^{-1}) \leq L(x)  + L(k) \leq r$. 
   This immediately yields
   \begin{align*}
       m_H (B_L(r-r_K)) \leq \vert P \cap B_L(r) \vert \cdot m_H(K),
   \end{align*}
   which we may rewrite as
   \begin{align*}
       \vert P \cap B_L(r) \vert \geq \frac{m_H (B_L(r-r_K))}{m_H(K)}.
   \end{align*}
   For $r \geq R_0 := \max \{r_V, 2r_K\}$, we have $r-r_K \geq r/2$ and thus
   \begin{align*}
        2r + r_V \leq 3r \leq 8(r-r_K).
   \end{align*}
   With $C_H \geq 1$ the doubling constant for the Haar measure $m_H$ on $H$, it follows that
   \begin{align*}
       m_H (B_L(2r + r_V)) \leq C_H^3 \cdot m_H (B_L(r-r_K)),
   \end{align*}
   which we combine with the above to obtain
   \begin{align*}
       \vert P \cap B_L(2r) \vert \leq \frac{m_H(B_L(2r + r_V))}{m_H(V)} \leq C_H^3 \cdot \frac{m_H(B_L(r-r_K))}{m_H(V)} \leq C_H^3 \frac{m_H(K)}{m_H(V)} \vert P \cap B_L(r) \vert,
   \end{align*}
   showing that we have doubling at scales larger than $\max \{r_V, 2r_K\}$. For scales smaller than this, that is $0 \leq r < R_0$, we have
   \begin{align*}
       \vert P \cap B_L(2r) \vert &\leq \vert P \cap B_L(2R_0) \vert \leq \frac{m_H(B_L(2R_0 + r_V))}{m_H(V)} \\
       &\leq \frac{m_H(B_L(2R_0 + r_V))}{m_H(V)} \cdot \vert P \cap B_L(r)\vert,
   \end{align*}
   where the last inequality follows from $\vert P \cap B_L(r)\vert \geq 1$ because $e \in P$. We therefore have doubling
   \begin{align*}
       \vert B_\ell (2r) P \vert \leq C \vert B_\ell(r) P \vert 
   \end{align*}
   using 
   \begin{align*}
       C \geq \max \{ C_H^3 \cdot \frac{m_H(K)}{m_H(V)}, \frac{m_H(B_L(2R_0 + r_V))}{m_H(V)} \}
   \end{align*}
   independently of $P \in G(\Lambda)^{(0)}$. This finishes the proof. 
\end{proof}

\begin{definition}\label{def:FLC-aperiodicity-repetitivity-minimality}
Suppose $H$ is a second-countable  locally compact group, and suppose $\Lambda \subseteq H$ is a separated point set. 
We say that $\Lambda$ has \emph{finite local complexity} (FLC) if for any compact $K \subseteq H$ the set $\{(x^{-1}\Lambda) \cap K \mid x \in \Lambda \}$ is finite, and $\Lambda$ is \emph{aperiodic} if $\{x \in H \mid x\Lambda = \Lambda\} = \{e\}$. 
\end{definition}
We record that finite local complexity passes to the hull. Indeed, if $\Lambda$ has FLC, then $\Delta:=\Lambda^{-1}\Lambda$ meets every compact set $K$ in a finite set: $\Delta\cap K$ is the union of the finitely many finite sets $(x^{-1}\Lambda)\cap K$, $x\in\Lambda$.
Thus $\Delta$ is closed and discrete. If $P\in\Omega(\Lambda)$, choose $h_n\Lambda\to P$. For $p,q\in P$, Chabauty--Fell convergence gives $p_n,q_n\in h_n\Lambda$ with $p_n\to p$ and $q_n\to q$.
Then $p_n^{-1}q_n\in\Delta$, so $p^{-1}q\in\Delta$ by closedness.
Hence $P^{-1}P\subseteq\Delta$, and every set $(p^{-1}P)\cap K$ is a subset of the finite set $\Delta\cap K$. This proves that $P$ has FLC. 
The proof of Lemma \ref{lemma:separating-and-dense-subsets-Delone} also shows that every element of the hull of a Delone set is Delone.

\begin{lemma}\label{lemma:minimality-Delone-point-set-gpd}
    Suppose $\Lambda \subseteq H$ is a separated point set (e.g. a Delone set). If the hull action $H \acts \Omega(\Lambda)$ is minimal, then $G(\Lambda)$ is a minimal groupoid. 
\end{lemma}

\begin{proof}
    By \cite[Proposition A.1]{EnstadRaumDynSampling2025} the set $\Omega_0(\Lambda)$ is an abstract transversal for $H \ltimes \Omega^\times (\Lambda)$, and so by \cite[Example 2.7]{MuhlyRenaultWilliamsJOT87}, $H \ltimes \Omega^\times (\Lambda)$ is equivalent to $(H \ltimes \Omega^\times (\Lambda))_{\Omega_0(\Lambda)} = G(\Lambda)$. Since minimality passes through groupoid equivalence \cite[Corollary 4.3]{vanWykWilliamsJOT24} we see that $G(\Lambda)$ is a minimal groupoid. 
\end{proof}

\begin{proposition}\label{prop:When-Delone-ample}
    Suppose $H$ is a second-countable locally compact group, and suppose $\Lambda \subseteq H$ is a Delone set which  has finite local complexity. Then $\Omega_0 (\Lambda)$ is zero-dimensional. If additionally $H$ is non-compact, the hull action $H \acts \Omega(\Lambda)$ is minimal and $\Lambda$ is aperiodic, then $\Omega_0 (\Lambda)$ has no isolated points. It is then a Cantor space.
\end{proposition}

\begin{proof}
    Put $\Delta=\Lambda^{-1}\Lambda$. As shown above, $\Delta$ meets compact sets in finite sets and contains every $Q\in\Omega_0(\Lambda)$. For $a\in\Delta$, choose an open neighbourhood $W_a$ with $W_a\cap\Delta=\{a\}$. Then $\{Q:a\in Q\}=\{Q:Q\cap W_a\neq\emptyset\}$ is open in the Chabauty--Fell topology, and its complement $\{Q:Q\cap\{a\}=\emptyset\}$ is also open. Thus $Q\mapsto\mathbf{1}_Q$ embeds the compact space $\Omega_0(\Lambda)$ into $\{0,1\}^{\Delta}$. The resulting clopen neighbourhood basis at $P$ consists of the sets
    \begin{align*}
        U_K=\{Q\in\Omega_0(\Lambda)\mid P\cap K=Q\cap K\},\qquad K\subseteq H\text{ compact}.
    \end{align*}
    In particular, $\Omega_0(\Lambda)$ is zero-dimensional.

     Now suppose further that $H$ is non-compact, the hull action $H \acts \Omega(\Lambda)$ is minimal and $\Lambda$ is aperiodic. 
    By Lemma \ref{lemma:minimality-Delone-point-set-gpd}, $G(\Lambda)$ is minimal.
    If its compact unit space had an isolated point, minimality and \'etaleness would make the unit space discrete and hence finite.
    However, noncompactness of $H$ and relative density of $\Lambda$ imply that $\Lambda$ is infinite, while aperiodicity makes
    \begin{align*}
        \Lambda\longrightarrow\Omega_0(\Lambda),\qquad
        \lambda\longmapsto\lambda^{-1}\Lambda
    \end{align*}
    injective. Thus $\Omega_0(\Lambda)$ is infinite and cannot have isolated points.

    Under all the assumptions it then follows that $\Omega_0(\Lambda)$ is a second-countable compact Hausdorff space which is zero-dimensional and without isolated points, i.e., a Cantor space.
\end{proof}

\begin{corollary}\label{cor:comparison-and-AH-Delone}
    Suppose $H$ is a second-countable locally compact group equipped with a Haar measure $m_H$ and a continuous proper length function $L$ for which \eqref{eq:bounded-Haar-volume-doubling} is satisfied, and suppose $\Lambda \subseteq H$ is a Delone set. 
    If  the hull action is minimal and $\Lambda$ has finite local complexity, then $G(\Lambda)$ has comparison, and the type semigroup $S(G(\Lambda))$ is almost unperforated. If furthermore $H$ is noncompact and $\Lambda$ is aperiodic, then $G(\Lambda)$ satisfies Matui's AH conjecture.
\end{corollary}

\begin{proof}
    Note first that since $H$ is second-countable, $\Omega_0 (\Lambda) = G(\Lambda)^{(0)}$  is second-countable, compact and Hausdorff, hence metrizable. 
    Then, since $\Lambda$ is a Delone set with finite local complexity, $G(\Lambda)$ is ample by Proposition \ref{prop:When-Delone-ample}. 
    Since $\ell$ given by \eqref{eq:def-length-ptset-gpd} is continuous and proper, it follows that $\ell$ is comparison admissible by Lemma \ref{lemma:comp-admissible-ample-existence}.
    Moreover, Proposition \ref{prop:fiberwise-doubling-Delone-pointset-gpd} yields uniformly bounded doubling for the same length function. As $G(\Lambda)$ is a subspace of $H \times \Omega_0 (\Lambda)$, its topology is second-countable. It is also $\sigma$-compact, as $G(\Lambda) = \cup_{n\geq 1} B_\ell(n)$, and each ball is compact. Since a compact ball is covered by finitely many open bisections, its intersection with any source fibre has uniformly bounded cardinality. Combining this with continuity, we observe that $\ell$ is uniformly fibrewise coarse. Under the same hypotheses, we note that $G(\Lambda)$ admits an invariant probability measure by Lemma \ref{lemma:density-approx}.
    Thus $G(\Lambda)$ has weak $m$-comparison for some sufficiently large $m$ by Theorem \ref{thm:main-theorem-wocg}. 
    Assuming the hull action is minimal, $G(\Lambda)$ is a minimal groupoid by Lemma \ref{lemma:minimality-Delone-point-set-gpd}, and it follows that $G(\Lambda)$ has comparison by Lemma \ref{lemma: weak m-comp implies comparison}. 
    Corollary \ref{cor:almost-unperforation} then yields that the type semigroup of $G(\Lambda)$ is almost unperforated. If moreover $H$ is noncompact and $\Lambda$ is aperiodic, then $\Omega_0 (\Lambda)$ has no isolated points by Proposition \ref{prop:When-Delone-ample}, and so it follows from \cite[Corollary E]{LiInfLoopSpaces25} that $G(\Lambda)$ satisfies Matui's AH conjecture. 
\end{proof}

\vspace{1cm} \noindent {\bf Use of AI assistance:} 
The writing of this paper was assisted by ChatGPT-6.0 Astra, a large language model made by OpenAI. After the authors discovered some gaps in the first version of the article and found potential fixes to them, the AI was used to help in adapting the old version to the current version. 
Furthermore, during the writing of this article, it 
suggested a number of changes to the exposition, some of which have been implemented. All AI-produced suggestions have been verified, and where necessary corrected, by the authors.

\printbibliography

\end{document}